\documentclass[10pt]{amsart}
\usepackage{amsxtra}
\usepackage[all]{xy}
\usepackage{palatino}
\usepackage{commath}
\usepackage{mathrsfs}
\usepackage{comment}
\usepackage{mathtools}

\usepackage{tikz-cd}
\usepackage[bookmarks=false]{hyperref}

\usepackage{amssymb,amsmath,amsthm,epsf,epsfig,dsfont,bbm}
\usepackage{upref, eucal}

\begin{document}

\newcommand\m{\mathfrak{m}}
\newcommand\J{\mathfrak{J}}
\newcommand\A{\mathbb{A}}
\newcommand\C{\mathbb{C}}
\newcommand\G{\mathbb{G}}
\newcommand\N{\mathbb{N}}
\newcommand{\T}{\mathbb{T}}
\newcommand{\cT}{\mathcal{T}}
\newcommand\sO{\cO}
\newcommand\sE{{\mathcal{E}}}
\newcommand\tE{{\mathbb{E}}}
\newcommand\sF{{\mathcal{F}}}
\newcommand\sG{{\mathcal{G}}}
\newcommand\GL{{\mathrm{GL}}}
\newcommand{\HH}{\mathrm H}
\newcommand\mM{{\mathrm M}}
\newcommand\fS{\mathfrak{S}}
\newcommand\fP{\mathfrak{P}}
\newcommand\fQ{\mathfrak{Q}}
\newcommand\Qbar{{\bar{\Q}}}
\newcommand\sQ{{\mathcal{Q}}}
\newcommand\sP{{\mathbb{P}}}
\newcommand{\Q}{\mathbb{Q}}
\newcommand{\tH}{\mathbb{H}}
\newcommand{\Z}{\mathbb{Z}}
\newcommand{\R}{\mathbb{R}}
\newcommand{\F}{\mathbb{F}}
\newcommand{\D}{\mathfrak{D}}
\newcommand\gP{\mathfrak{P}}
\newcommand\Gal{{\mathrm {Gal}}}
\newcommand\SL{{\mathrm {SL}}}
\newcommand\Spec{{\mathrm {Spec}}}
\newcommand\Hom{{\mathrm {Hom}}}
\newcommand{\legendre}[2] {\left(\frac{#1}{#2}\right)}
\newcommand\iso{{\> \simeq \>}}
\newcommand{\cX}{\mathcal{X}}
\newcommand{\cJ}{\mathcal{J}}
\newtheorem{theorem}{Theorem}
\newtheorem{corollary}{Corollary}[theorem]
\newtheorem{thm}{Theorem}
\newtheorem{cor}[thm]{Corollary}
\newtheorem{conj}[thm]{Conjecture}
\newtheorem{prop}[thm]{Proposition}
\newtheorem{lemma}[thm]{Lemma}
\theoremstyle{definition}
\newtheorem{definition}[thm]{Definition}
\theoremstyle{remark}
\newtheorem{remark}[thm]{Remark}
\newtheorem{example}[thm]{Example}
\newtheorem{claim}[thm]{Claim}
\newtheorem{lem}[thm]{Lemma}

\theoremstyle{definition}
\newtheorem{dfn}{Definition}

\theoremstyle{remark}

\theoremstyle{remark}
\newtheorem*{fact}{Fact}
% type user-defined commands here
\makeatletter
\def\imod#1{\allowbreak\mkern10mu({\operator@font mod}\,\,#1)}
\makeatother
\newcommand{\mcF}{\mathcal{F}}
\newcommand{\mcG}{\mathcal{G}}
\newcommand{\mcM}{\mathcal{M}}
\newcommand{\cP}{\mathfrak{P}}
\newcommand{\mcP}{\mathcal{P}}
\newcommand{\mcS}{\mathcal{S}}
\newcommand{\mcV}{\mathcal{V}}
\newcommand{\mcW}{\mathcal{W}}
\newcommand{\gS}{\mathcal{S}}
\newcommand{\cO}{\mathcal{O}}
\newcommand{\cC}{\mathcal{C}}
\newcommand{\mfa}{\mathfrak{a}}
\newcommand{\mfb}{\mathfrak{b}}
\newcommand{\mfH}{\mathfrak{H}}
\newcommand{\mfh}{\mathfrak{h}}
\newcommand{\mfm}{\mathfrak{m}}
\newcommand{\mfn}{\mathfrak{n}}
\newcommand{\mfp}{\mathfrak{p}}
\newcommand{\mfq}{\mathfrak{q}}
\newcommand{\cF}{\mathfrak{F}}
\newcommand{\mrB}{\mathrm{B}}
\newcommand{\mrG}{\mathrm{G}}
\newcommand{\gG}{\mathcal{G}}
\newcommand{\mrH}{\mathrm{H}}
\newcommand{\mH}{\mathrm{H}}
\newcommand{\mrZ}{\mathrm{Z}}
\newcommand{\Ga}{\Gamma}
\newcommand{\cyc}{\mathrm{cyc}}
\newcommand{\Fil}{\mathrm{Fil}}
\newcommand{\mrp}{\mathrm{p}}
\newcommand{\PGL}{\mathrm{PGL}}
\newcommand{\x}{{\mathcal{X}}}
\newcommand{\Sp}{\textrm{Sp}}
\newcommand{\ab}{\textrm{ab}}

\newcommand{\lra}{\longrightarrow}
\newcommand{\ra}{\rightarrow}
\newcommand{\rai}{\hookrightarrow}
\newcommand{\ras}{\twoheadrightarrow}

\newcommand{\repr}{\rho_{f,\wp}|_{G_p}}
\newcommand{\GRF}{{\rho}_{f,\wp}}

\newcommand{\lan}{\langle}
\newcommand{\ran}{\rangle}

\newcommand{\mo}[1]{|#1|}

\newcommand{\hw}[1]{#1+\frac{1}{2}}
\newcommand{\mcal}[1]{\mathcal{#1}}
\newcommand{\trm}[1]{\textrm{#1}}
\newcommand{\mrm}[1]{\mathrm{#1}}
\newcommand{\car}[1]{|#1|}
\newcommand{\pmat}[4]{ \begin{pmatrix} #1 & #2 \\ #3 & #4 \end{pmatrix}}
\newcommand{\bmat}[4]{ \begin{bmatrix} #1 & #2 \\ #3 & #4 \end{bmatrix}}
\newcommand{\pbmat}[4]{\left \{ \begin{pmatrix} #1 & #2 \\ #3 & #4 \end{pmatrix} \right \}}
\newcommand{\psmat}[4]{\bigl( \begin{smallmatrix} #1 & #2 \\ #3 & #4 \end{smallmatrix} \bigr)}
\newcommand{\bsmat}[4]{\bigl[ \begin{smallmatrix} #1 & #2 \\ #3 & #4 \end{smallmatrix} \bigr]}
\newcommand\mycom[2]{\genfrac{}{}{0pt}{}{#1}{#2}}

\makeatletter
\def\imod#1{\allowbreak\mkern10mu({\operator@font mod}\,\,#1)}
\makeatother

\title{Two results on Eisenstein part of homology and cohomology groups of Bianchi modular groups}

\author{Debargha Banerjee}
\author{Pranjal Vishwakarma}
\address{INDIAN INSTITUTE OF SCIENCE EDUCATION AND RESEARCH, PUNE, INDIA}
\thanks{The first named author was partially supported by the SERB grant CRG/2020/000223, and the second named author by the PMRF fellowship. Authors are greatly indebted to Professors Eknath Ghate, Haruzo Hida, and Haluk Seng\"un for several email correspondences.   The article is dedicated to the fond memory of Professor Yuri Manin.}

\subjclass[2010]{Primary: 11F67, Secondary: 11F41, 11F20, 11F30}
\keywords{Eisenstein series, Modular symbols, Special values of $L$-functions}
 \setcounter{tocdepth}{1}
  
 \begin{abstract}
We explicitly write down the {\it Eisenstein cycles} in the first homology groups of quotients of the hyperbolic three spaces as linear combinations of Cremona symbols (a generalization of Manin symbols) for imaginary quadratic fields.   They generate the Eisenstein part of the homology groups. We also study the Eisenstein part of the cohomology groups. As an application, we find an asymptotic dimension formula (level aspect)
for the cuspidal cohomology groups of congruence subgroups of the form $\Ga_1(N)$ inside the full {\it Bianchi groups}.  
\end{abstract}

 \maketitle

\section{Introduction}
\label{Intro}

\subsection{Introduction}
Eisenstein part of the homology and cohomology groups for {\it classical} modular curves encodes several arithmetic information, such as special values of $L$-functions. Eisenstein ideals were introduced by Mazur, and these ideals found applications in several important breakthroughs in the last few decades, including Fermat's last theorem and the main conjecture of Iwasawa theory.

Merel introduced Eisenstein cycles that are the basis of the {\it integral} homology groups of classical elliptic modular curves (space of modular symbols).
Banerjee-Merel~\cite{MR3873111} investigated Eisenstein cycles for {\it congruence} subgroups of classical modular groups.

In the classical case for the congruence subgroups of odd level, we can take intersection with 
the principal congruence subgroup of level two. As a consequence, we can assume all the cusps are lying above $\{0, \infty, 1\}$ and cycles can be taken of the form $\{g1,g(-1)\}$ for $g \in \SL_2(\Z)$. In turn, these can be written in terms of loops of the modular curves. 

For any odd integers $N$, Banerjee-Merel~\cite{MR3873111} expressed Eisenstein cycles in the first homology groups of modular curves of level $N$ as linear combinations over $\overline{\Q}$ of Manin symbols (generators of homology groups of elliptic modular curves). The coefficients are computed in terms of periods (integrals over cycles of complex differential $1$ forms associated to the Eisenstein series). These periods can be computed in terms of Dedekind sums.

  This formulation extended to subgroups of finite indices (not necessarily congruence subgroups)  within the full modular group as detailed in~\cite{banerjee2022eisenstein} with a more complicated notion of periods.  It should be noted that these Eisenstein cycles are {\it not} necessarily  $\overline{\Q}$ valued for general subgroups of finite index.

John Cremona and his collaborators initiated the study of modular symbols (integral homology groups of quotients of three-dimensional hyperbolic space) for imaginary quadratic fields. They are mostly interested in the cuspidal part of homology groups of the corresponding topological space.  Ito \cite{MR870309} first defines periods in this setting and showed that the periods are again described by Sczech's Dedekind sums \cite{scz} for $\Gamma = \SL_2(\cO_K)$.

In this paper, we write down the Eisenstein cycles of the homology group associated 
to the quotient of hyperbolic three-space by subgroups of finite index of Bianchi modular groups in terms of Cremona symbols (a generalization of Manin symbols for Bianchi groups). These Eisenstein cycles generate the Eisenstein part of the homology groups. 
These Eisenstein cycles are determined 
in terms of generalized periods (integrals of two forms over a real surface) that are hard to compute as of now.

In the present setting, results are difficult to obtain because
the corresponding topological spaces do {\it not} carry the same
algebraic geometric structure as in the cases for elliptic or Hilbert modular forms settings. 
In particular, there is no analogue of dessin d'enfant.

We study the Eisenstein part of the cohomology groups and, as an application 
find an asymptotic dimension formula (level aspect) for the space of cuspidal cohomology groups. 

Till date, there is no nice formula for the dimensions of the space Bianchi cusps forms similar to elliptic cases. Note that we can not apply the Riemann-Roch theorem 
in this setting to obtain a dimension formula. 
Seng\"un and his collaborators ~\cite{MR3464620} studied the dimensions of the cohomology of Bianchi groups for the {\it principal} congruence subgroups.

A natural question arises regarding the generalization of these computations to other crucial congruence subgroups such as $\Gamma_1(\mfa)$ as defined below. To study Diophantine questions like modularity, Bianchi modular $3$ fold associated to these subgroups are expected to play an important role.  As an application of our study of Eisenstein part of the cohomology groups of subgroups of Bianchi modular groups, we give a {\it lower} bound of the dimension of corresponding cohomology groups for the congruence subgroups of the form $\Ga_1(p^n)$ a $n \rightarrow \infty$. Note that this part contains corresponding Bianchi cusp forms by the Harder-Eichler-Shimura isomorphism theorem (generalized Matsushima's formula). These results show that at least these many Bianchi cusp forms will be 
there for the congruence subgroups of the form $\Ga_1(p^n)$.

We now explain our result. 
Let $\cO_K$ be the ring of integer of an imaginary quadratic field $K$ of class number $h(K)=1$. Let
$\tH_3:=\{(z,t): z\in \C, t\in \R^+ \}$ be the hyperbolic $3$-space and let $ \Gamma \leq \SL_2(\cO_K):=G$ be a subgroup of the finite index with no elements of finite order. Let $P:=\left\{\pm\left[\begin{array}{cc}1 & j \\ 0 & 1\end{array}\right]: j \in \cO_K \right\}$ be the parabolic subgroup of $\mathrm{SL}_2(\cO_K)$.
Consider the Bailey-Borel-Satake compactification $X^{BB}_\Gamma = \Gamma \backslash \overline{\tH}_3$ obtained by adding the set of cusps~\cite{MR4569265} and we study the corresponding homology groups. To compute the cohomology group, we also study the Borel-Serre compactification $X^{BS}_\Gamma$
obtained by adding a $2$ torus to each cusps $\partial(X^{BS}_\Gamma)$ (except for $K=\Q(i)$ or $K=\Q(\sqrt{-3})$ for which we add spheres instead).

 The space of $1$-differential forms on $\tH_3$ is given by the basis
\[
\beta = (\beta_0,\beta_1,\beta_2) = (-\frac{dz}{t},\frac{dt}{t},\frac{d\bar{z}}{t}).
\]
Let $E_2(\Ga)$ be the space of Bianchi Eisenstein modular form of weight two for a subgroup $\Ga \leq \SL_2(\sO_K)$ of finite index. 
For $E \in E_2(\Ga)$, denote the differential $1$-form associated 
to $E$ as $\omega_E$ and consider the differential form $\star \omega_E$ (cf. \S~\ref{Eisendf}).

\begin{definition}
The {\it Eisenstein classes} are modular symbols $\sE \in \HH_1(X^{BB}_\Gamma, \partial(X^{BB}_{\Gamma});  \C)$ such that 
\[
\int_{\sE} F \cdot \beta=0
\]
for all  cusp forms $F$ (see Definition ~\ref{cuspbianchi} for the definition of the cusp forms). The {\it Eisenstein cycles} are then paths in the Eisenstein classes. 
\end{definition}
We prove that the boundaries of Eisenstein cycles are non-zero under certain assumptions for restricted class of subgroups.
It is expected that Eisenstein classes will have certain Hecke equivariant property for congruence subgroups similar to classical 
situation (see also ~\cite[Theorem 2]{MR3873111}). 
The aim of this article is to find explicit Eisenstein classes that are generators of the Eisenstein part of the Bianchi modular symbols. In other words, 
we explicitly write an Eisenstein cycle $\sE_{E}$ corresponding to an Eisenstein series $E$. Cremona ~\cite{MR743014} showed that (see also ~\cite[thm 4.3.2]{JCthesis})
the map 
\[
\eta: (g)= \{g 0, g \infty\}
\]
is a surjective map from $\Gamma \backslash \SL_2(\cO_K)$ onto $\HH_1(X^{BB}_\Gamma, \partial(X^{BB}_\Gamma);\Q)$.  The kernel of the map can be computed as in loc. cit. 

 Let  $\star$ be the Hodge star operator acting on differential forms. Let $\cF_K$ be
 fundamental domain for the imaginary number field $K$ with boundary 
 $\partial(\cF_K)$ ~\cite{aranes2010modular}. Define the generalized $2$ periods: 
\[
F_E(g):=  \int\limits_{g\cF_K} d ( \star   \omega_E). 
\]

Let $E_2(\Ga)$ be the space of Eisenstein modular form of weight two as defined in \S~\ref{Bianchidefinition}. 
For the next result, we assume that $K=\Q(\sqrt{-d})$ is an Euclidian domain. In other words, we assume $d \in \{1,2,3,7,11\}$. 
We now state our first theorem. 
\begin{thm}
\label{Main-thm}
For any imaginary field $K$ with class number one that is also an Euclidean domain,  the modular symbol
$$\sE_{E}= \sum_{g \in \Gamma \backslash G} F_E(g) \eta (g)$$
is the Eisenstein element corresponding to the Eisenstein series $E \in E_2(\Gamma)$.
\end{thm}
Similar to the classical case ~\cite{banerjee2022eisenstein}, we cannot determine if these Eisenstein cycles are $\overline{\Q}$ valued. Of course, we expect 
these Eisenstein cycles to be $\overline{\Q}$ valued only for the Eisenstein series for the congruence subgroups. We produce an explicit element of the Eisenstein part of the modular symbols for every cusp. They generate the Eisenstein part of the modular symbols.

The above calculations are for the homology group. We now study the Eisenstein part of the cohomology groups since there are close relations between homology and cohomology groups by various dualities, and cohomology groups are sometimes more amenable for calculations. We study the Lefschetz numbers for the action of complex conjugation on the cohomology groups. For this part, we use Borel-Serre compactification. We prove an asymptotic dimensional formula for the cuspidal cohomology groups for congruence subgroups of the form $\Ga_1(\mfa)$. Let $\mfa$ be an ideal in $\cO_K$. Define a congruence subgroup of level $\mfa$,
\[
\Ga_1(\mfa) := \Big\{\begin{pmatrix}a&b\\ c&d\end{pmatrix}\in \SL_2(\cO_K):a-1,d-1, c\in \mfa \Big\}.
\]

Let $\HH_{cusp}^1\left(\Gamma_1\left(p^n\right); \mathbb{C}\right) $
be the cuspidal part of the cohomology
group (cf. Definition~\ref{Eisensteincuspidal}). This is the complement 
of the Eisenstein part of the cohomology group. By generalized 
Eichler-Shimura-Harder theorem, these cohomology groups contain 
cusp forms.

We calculate the level aspect of asymptotic lower bounds for the dimension of the
cuspidal cohomology in  Theorem ~\ref{dimensionasymptotice} for the congruence subgroup of the form $\Ga_1(N)$ of the Bianchi modular groups (cf. Seng\"un-T\"urkelli ~\cite{MR3464620} for $\Ga(N)$). 
\begin{thm} 
\label{dimensionasymptotice}
Let $p$ be a rational prime that is unramified in $K$ and let $\Gamma_{1}\left(p^n\right)$ 
denote the  subgroup of level $(p^n)$ of a Bianchi group $\mathrm{SL}_2(\cO_K)$ for all $n \geq 1$. Assume further that the class number of $K$ is one.
Then we have the following asymptotic dimension (level aspect) of the cuspidal cohomology groups:
$$
\operatorname{dim} \HH_{cusp}^1\left(\Gamma_1\left(p^n\right); \mathbb{C}\right) \gg p^{3n}
$$
as $n$ increases.
\end{thm}
The proof of the above theorem is similar to that of Seng\"un-T\"urkelli ~\cite{MR3464620}. However, we note that in Proposition ~\ref{Lefschetzkey}, two important quantities appear in the Lefschetz trace formula. These quantities are known 
for the principal congruence subgroups of the Bianchi group by the work of Rohlf \cite{MR507801}. We need to compute the same in our setting in Lemma~\ref{AB}. 
We also need to compute the order of the cuspidal subgroup for $\Ga_1(N)$ to prove our Theorem~\ref{dimensionasymptotice}. The key ingredient is Theorem~\ref{cohomology1} regarding Lefschetz numbers proved in \S~\ref{Tracecohomo}. A slightly different quantity appears in this statement of this theorem for the congruence subgroup of the form $\Ga_1(N)$. 
 
\section{Bianchi modular symbols and Cremona symbols}
\label{Hecke}

\subsection{Fundamental domains} \label{2.1}
Recall that 
\[
\tH_3=\{(z,t): z\in \C, t\in \R^+ \} \simeq \GL_2(\C)/Z\cdot \mathrm{SU}(2)
\]
is the hyperbolic $3$ space, where $Z$ denotes the diagonal matrices in $\GL_2(\C)$. An invariant metric for the action of $\GL_2(\C)$ on $\tH_3$ is then given by 
\[
ds = \frac{(dzd\bar{z} + (dt)^2)}{t^2}.
\]
We say that a subgroup $\Gamma \leq \SL_2(\C)$ acts discontinuously on $\tH_3$ if every compact subset of $\tH_3$ meets only a finite number of its images under $\Ga$.

A subgroup $\Gamma$ acts discontinuously if and only if it is discrete (in the matrix topology), and we define a fundamental domain or region for $\Gamma$ to be an open subset $\cF \subset \tH_3$ such that every orbit of $\Ga$ in $\tH_3$ meets $\cF$ at most once and meets the closure of $\cF$ at least once.

Choose a fundamental region $\cF$  defined in ~\cite[ \S 2.3]{JCthesis} for the action of $\Gamma$ on $\overline{\tH}_3$ with $\{0,j\infty\}$ as one of its edges. Let $\Ga_P$ be the stabilizer of a vertex $P$ of $\cF$. Form a larger basic polyhedron by taking the union of the transforms of $\cF$ by a finite subgroup $\Ga_P$. 
See also the work of Arenes~\cite{aranes2010modular},  Lingham~\cite{lingham2005modular}

For example,  we are first interested in the case $G=\SL_2(\Z[i])$. This subgroup
is generated by 
\[
S=\begin{pmatrix}0&-1\\ 1&0\end{pmatrix},\
T=\begin{pmatrix}1&1\\ 0&1\end{pmatrix},\
U=\begin{pmatrix}1&i\\ 0&1\end{pmatrix},\
R=\begin{pmatrix}0&i\\ i&0\end{pmatrix},
\]
with the relations $TU=UT, S^2=R^2=(RS)^2=(URS)^2=(TS)^3=(UR)^3=1$. 
Consider three points on the unit sphere $P_1=(\frac12,0,\frac12\sqrt3), P_2=(\frac12,\frac12,\frac12\sqrt2),$ and $P_3=(0,\frac12,\frac12\sqrt3)$. A fundamental domain for the action of $\Ga$ is given by vertices at $0,\infty$, and $P_1, P_2 $ and $P_3$.  
The stabilizer of $P_2$ is $\Ga_P=\langle TS, UR\rangle,$ which is of order 12 ~\cite[p.59]{JCthesis}.

We call the union of the $12$ translates of $\cF$ by $\Ga_P$ the basic polyhedron, denoted $B$. It is a hyperbolic octahedron, and its edges $12$ are precisely the images of $\{0,\infty\}$ under the action of $\Ga_P$. We refer to loc. cit.  for the diagram showing the region $\cF$ and the octahedron $B$. Note that $B$ has a triangular face whose edges are the transforms of $\{0,\infty\}$ under $I,TS,$ and $(TS)^2$, while the images of $\{0,\infty\}$ under $UR, (UR)^2$ are $\{\infty,i\}$ and $\{i,0\}$ respectively.

For the field $K= \mathbb{Q}(i)$, the fundamental domain $\cF$ is defined in ~\cite[ \S 2.3]{JCthesis}. The boundary of the fundamental domain, denoted as $\partial{\cF}$, is a combination of six faces given by  $B_1=\{0, P_1, \infty\}$,  $B_2=\{0, P_3, \infty\}$, $B_3=\{P_2, P_3, \infty\}$, $B_4=\{P_1, P_2, \infty\}$, $B_5=\{0, P_1, P_2\}$, and $B_6=\{0, P_2, P_3\}$.
 \label{boundary}
Integration over the boundary of the fundamental domain is the sum of integrations over each face.

\section{Differential forms} \label{Differential forms}
The space of real $1$-differential forms on $\tH_3$ is given by the basis
\[
\beta = (\beta_0,\beta_1,\beta_2) = (-\frac{dz}{t},\frac{dt}{t},\frac{d\bar{z}}{t}),
\]
and denote the pullback of each $\beta_i$ to $\GL_2(\C)$ by $\omega_i$. A differential form on $\GL_2(\C)$ is the inverse image of a differential form on $\tH_3$
if and only if it can be written as $\varphi_0\omega_0+\varphi_1\omega_1+\varphi_2\omega_2$.  Here,  
$\Phi=(\varphi_0,\varphi_1,\varphi_2)$ satisfies $\Phi(gkz)= \Phi(g)\rho(kz)$ for every $g\in \GL_2(\C), k\in \mathrm{SU}_2(\C), z\in Z$, and $\rho$ is a fixed representation of $\mathrm{SU}_2(\C)$.

Let $F=(F_0,F_1,F_2)$ be a function on $\tH_3$, and $F\cdot\beta = F_0\beta_0+F_1\beta_1+F_2\beta_2$ a $1$-form on $\tH_3$.
Recall that $*$ is the usual adjoint operator of differential forms.  By definition, one has
$*(F\cdot\beta) = -\frac12i\bar{F}_1(\beta_0\wedge\beta_2) +i\bar{F}_0(\beta_1\wedge\beta_2) +i\bar{F}_2(\beta_0\wedge\beta_1)$.
The differential form is harmonic if and only if $F\cdot\beta$ and $*(F\cdot\beta)$ are closed forms. 

Observe that 
\[
d \left( \star (F \cdot \beta)\right)=i H_F \beta_0 \wedge \beta_1 \wedge \beta_2;
\]
with $H_F=t\left( \frac{\partial \bar{F}_1}{\partial z}+ \frac{1}{2}\frac{\partial \bar{F}_0}{\partial t}+ \frac{\partial \bar{F}_2}{\partial \bar{z}}\right).$
\begin{definition}
\label{vectorhar}
(Vector-valued harmonic differential functions)
We say that a function $F:\tH_3\to\C^3$ is harmonic if $F\cdot\beta$ is a harmonic differential form and is slowly increasing in the sense that there exists $N\geq 0$ such that.
\[
F\Big(\begin{pmatrix}x&0\\0&1\end{pmatrix}(z,t)\Big)=O(|x|^N), \quad x\in\R
\]
as $x\to\infty$ uniformly in compact sets in $\tH_3$. 
\end{definition}
\subsection{Bianchi modular forms}\label{Bianchidefinition}
 Let $\Delta=d \circ \delta+\delta \circ d$ be the usual Laplacian with $d$ being the exterior derivative and $\delta$ the co-differential operator on $\tH_3$.
Recall first that the differential form $\omega$ is harmonic if $\Delta \omega=0$.  
We now recall the weights of the Bianchi modular forms. The weights of Bianchi modular forms can be scalar-valued  ~\cite{JCthesis} or more generally, vector-valued ~\cite{Chriswilliams} (see also ~\cite{MR1671216}).

Now, we define scalar-valued {\it Bianchi modular form} of weight $2$. We wish to define a slashing operator such that $F \cdot \beta$ is an invariant differential under the action of a discrete subgroup $\Gamma$ of $\SL_2( \cO_K)$. The complex modular group $\SL_{2}(\C)$ acts on the space of differential $1$ -forms as
$\beta^{\prime}=\mathrm{J}(\gamma ; (z, t)) \beta$ 
   with 
 
 $ \mathrm{J}(\gamma ; (z, t))=\frac{1}{\left(|r|^2+|s|^2\right)}\left(\begin{array}{ccc}r^2  & -2 r s  & s^2  \\ \bar{r} \bar{s}& (r \bar{r}-s \bar{s}) & -\bar{r} s \\ \bar{s}{ }^2  & \overline{2 r s} & \bar{r}{ }^2 \end{array}\right)$, $\gamma=\left(\begin{array}{ll}a & b \\ c & d\end{array}\right) \in \SL_2(\mathbb{C})$, $r=\overline{c z+d}$ and $s=\bar{c}t$.

Let $F:\tH_3\to\C^{3}$ be a harmonic function. Define the slash operator.

\begin{center}
 ($F \mid_\gamma)(z, t)=F(\gamma(z, t)$) $\mathrm{J}(\gamma ; (z, t)$).
\end{center}

Now we define the {\it Bianchi modular form} of weight $2$. 
Let $F:\tH_3\to\C^{3}$ be a harmonic function. Define the slash operator
$$
\left(\left.F\right|_{\gamma} \right)(z,t):=Sym^{2}\left(j\left(\gamma;  (z,t)\right)^{-1}\right) F(\gamma (z,t)),
$$
where
\[
j\left(\gamma; (z,t)\right) = \begin{pmatrix}cz+d&-ct\\ \bar{c}t &\overline{cz+t}\end{pmatrix}, \qquad \gamma=\begin{pmatrix}a&b\\ c&d\end{pmatrix},
\]
and $Sym^2$ is the symmetric $2$-nd power of the standard  representation of $\SL_2(\C)$ on $\C^2$.

We have $F: \mathbb{H}_{3} \rightarrow \mathbb{C}^{3}$ and
$$
(F|_\gamma)(z,t)=\frac{1}{|r|^{2}+|s|^{2}}\left(\begin{array}{ccc}
{r}^{2} & 2 r\bar{s}  & \bar{s}^{2} \\
-{r} {s} & |r|^{2}-|s|^{2}  & \bar{r}  \bar{s} \\
{s}^{2} & -2 \bar{r} {s} & \bar{r}^{2}
\end{array}\right) F(\gamma (z,t)),
$$
where $\gamma=\left(\begin{array}{cc}a& b \\ c& d\end{array}\right) \in \SL_2(\mathbb{C})$ and $r=\overline{c z+d}$ and $s=\bar{c} t$. The $1$-forms $\beta_{0}:=-\frac{d z}{t}, \beta_{1}:=\frac{d t}{t}, \beta_{2}:=\frac{d \bar{z}}{t}$ form a basis of differential $1$-forms on $\mathbb{H}_{3}$. 

The modular group $\SL_{2}(\C)$ acts on the space of differential $1$ -forms as
$$
\gamma \cdot{ }^{t}\left(\beta_{0}, \beta_{1}, \beta_{2}\right)=Sym^{2}(j(\gamma; (z,t)))^{t}\left(\beta_{0}, \beta_{1}, \beta_{2}\right). 
$$
Here  $Sym^2$ is the symmetric $2$-nd power representation of $\SL_2(\C)$ on $\C^2$.

With both the definition of weights, the differential $F \cdot \beta$ is $\Ga$ invariant. We can also generalize vector-valued {\it Bianchi modular form} to arbitrary weight $k$ by changing $Sym^2$ to $Sym^k$. However, we are only interested in weight two in this paper. 

\begin{definition}
\label{cuspbianchi}
\begin{enumerate}
\item
 (Bianchi modular forms)
A weight $2$ {\it Bianchi modular modular form} for a congruence subgroup $\Gamma \leq \SL_2(\cO_K)$ is a real analytic vector valued  $\Gamma$ invariant function $F: \mathbb{H}_{3} \rightarrow \mathbb{C}^{3}$ (cf. Definition~\ref{vectorhar}). In other words, the function $F$
satisfies the invariance property  $F|_\gamma=F$ for all $\gamma \in \Ga$. We denote the set of all weight two Bianchi modular forms  by 
$M_2(\Gamma)$. 
\item
(Bianchi cusp forms)
If $F \in M_2(\Gamma)$ satisfies the additional properties  that 
$\int\limits_{\C / \cO_K} F \mid_\gamma(z,t) dz=0$ for every $\gamma \in \SL_2(\cO_K)$, we call it a Bianchi cusp form. 
\end{enumerate}
\end{definition}
The last condition is equivalent to saying that the constant coefficient in the Fourier-Bessel expansion of $F\mid_\gamma$ is equal to zero for every $\gamma\in \SL_{2}(\cO_K)$. 
The $\Ga$ invariance implies that $F$ has a Fourier-Bessel expansion of the form

$$
F(z, t)={\sum\limits_{\alpha \in \cO_K, \alpha \neq 0}c(\alpha) t^{2} \mathbf{K}\left(\frac{4 \pi|\alpha| t}{\sqrt{d_{K}}}\right) \psi\left(\frac{\alpha z}{\sqrt{d_{K}}}\right)};
$$

where $\psi(z) = e^{2\pi(z+\bar z)}$ and ${\bf K}(t) = (-\frac{i}{2}K_1(t),K_0(t),\frac{i}{2}K_{1}(t)))$ with $K_0,K_1$ being the modified Bessel functions satisfying the differential equation
\label{bessel}
$$
\frac{dK_j}{dt^2}+\frac{1}{t}\frac{dK_j}{dt}-\Big(1+\frac{1}{t^{2j}}\Big)K_j = 0, \qquad j=0,1
$$

decreasing rapidly at infinity.

Let  $S_2(\Gamma)$ be the space of weight $2$ Bianchi   modular forms for a subgroup $\Gamma$ of $\SL_2(\cO_K)$.
Let $F=(F_0, F_1, F_2)$ and $S=(S_0, S_1, S_2)$. By ~\cite[p. 549]{MR503433}, the inner product is given by 
\[
 <F, S>
  := \frac{1}{12i[G:\Gamma]}\int_{\Gamma \backslash \overline{\tH}_3} F \cdot \beta  \wedge *  (S \cdot \beta) 
\]
The set of Eisenstein Bianchi modular form $E_2(\Gamma)$ is the compliment of $S_2(\Gamma)$ inside $M_2(\Gamma)$. In other words, we have a decomposition 
\[
M_2(\Gamma)=S_2(\Gamma)\bigoplus E_2(\Gamma).
\]
If $\Ga$ is a congruence subgroup of the form $\Ga_1(N)$, dimension of the $E_2(\Gamma)$  can be computed using \cite[Proposition 4.4]{MR4356858}.

\section{Eisenstein differential forms}
\label{Eisendf}
\subsection{Ito’s differential forms ~\cite{MR870309} for full subgroup $\SL_2(\sO_K)$} \label{Ito}
We write elements of $\tH_3$ as quaternion numbers $u=z+j t \in \mathbb{H}_3$ 
with $j^2=-1, ij=-ji$. Write $z(u)=z$ and $t(u)=t$ if $u=z+j t$.
 A matrix $M=\left(\begin{array}{ll}a & b \\ c & d\end{array}\right) \in \SL_2(\sO_K)$ acts on $\mathbb{H}_3$ by
$$
M u:=(a u+b)(c u+d)^{-1},
$$
where the right-hand side is taken in the skew field of quaternions. For two complex numbers $(m,n) \neq (0,0)$, take a matrix 
$M=\left(\begin{array}{cc}* & * \\ m & n\end{array}\right) \in \SL_2(\sO_K)$ and let
\[
\begin{aligned}
&t(m, n ; u)=t(M u), \\
&J(m, n ; u)=\left(\frac{\partial}{\partial z}, \frac{\partial}{\partial t}, \frac{\partial}{\partial \bar{z}}\right) z(M u) .
\end{aligned}
\]
 The Eisenstein series
\[
E(u;s):=\sum_{m, n \in \sO_K}^{\prime} J(m, n ; u) t(m, n ; u)^s, \quad \operatorname{Re}(s)>1
\]
can be continued analytically to the whole $s$-plane. Here the prime on the summation symbol means to omit the meaningless terms,
i.e. the term corresponding to $m=n=0$. Define the complex-valued differential form $\omega_E$ on $\mathbb{H}_3$ by
\[
\omega_E=E(u, 0)\left(\begin{array}{l}
d z \\
d t \\
d \bar{z}
\end{array}\right) .
\]
From loc. cit, this differential form is closed and invariant with respect to $\SL_2(\cO_K)$.
 Let $D(K):=w_1 \bar{w}_2-\bar{w}_1 w_2$ if $\sO_K=\mathbb{Z} w_1+\mathbb{Z} w_2$ with $\operatorname{Im}\left(w_1 / w_2\right)>0$, and $K^{\prime}=D(K)^{-1} \bar{K}$. 
The indefinite integral of $\omega_E$ can be written in the form of the Fourier expansion.
This integrals produce periods iterms of the fuction
$$ 
 \widetilde{H}(u)=G_2(0)(z-\bar{z})-\frac{4 \pi}{D(K)} t \sum_{m \in K, n \in K^{\prime}}^{\prime} \frac{\bar{m} n}{|m n|} K_1(4 \pi|m n| t) e(m n z).
 $$
For each non-negative integer $k$, define Ito's Eisenstein series by generalized Hecke's summation tricks: 
\[
G_k(x):=\left.\sum_{w \in \sO_K}^{\prime}(w+x)^{-k}|w+x|^{-s}\right|_{s=0},
\]
where the value at $s=0$ should be understood in the sense of analytic continuation. 
We take the prime in the summation symbol to omit the terms corresponding to $m=n=0$. Furthermore, $e(z)=\exp (2 \pi i(z+\bar{z}))$ and $K_v(z)$ denotes the modified Bessel function of the second kind. 

Recall the following important theorem regarding the periods of the Eisenstein series in Bianchi modular forms settings. 
The function $\widetilde{H}$ below is harmonic with respect to the Riemannian structure of $\mathbb{H}_3$ given by the $\SL_2(\mathbb{C})$-invariant Riemannian metric $\frac{1}{t^2}\left(d x^2+d y^2+d t^2\right)$ $\quad(u=z+j t \in \mathbb{H}_3, z=x+i y)$.

\begin{theorem} \label{thm1}
Define the map $\Psi: \SL_2(\cO_K) \rightarrow \mathbb{C}$ by
$$
\Psi\left(\begin{array}{ll}
a & b \\
c & d
\end{array}\right)= \begin{cases}G_2(0) I(\frac{a+d}{c})-D(a, c), & c \neq 0, \\
G_2(0) I\left(\frac{b}{d}\right), & c=0\end{cases}
$$
with $I(z):=z-\bar{z}$ and
$$
D(a, c):=\frac{1}{c} \sum_{r \in K / c K} G_1\left(\frac{a r}{c}\right) G_1\left(\frac{r}{c}\right)
$$
Then
$$
\widetilde{H}(A u)=\widetilde{H}(u)+\Psi(A)
$$
for every $A$ in $\SL_2(\cO_K)$.
\end{theorem}

Let $\omega_E$ be as defined above, i.e

\[
\omega_E = E(u, 0)\left(\begin{array}{l}
d z \\
d t \\
d \bar{z}
\end{array}\right).
\]
By applying the $*$ operator as defined in ~\ref{Differential forms} to the differential form $\omega_E$, we can compute the $*$ operator of $\omega_E$.
\subsection{Generlization of Ito’s differential forms ~\cite{MR870309} for subgroups of $\SL_2(\sO_K)$} \label{Itosubgroups}
Let $\{\kappa_1,\dots,\kappa_h\}$ be a complete set of representatives of the cusps of $\Gamma$. For each $i$, denote by $\Gamma_i$ to be the stabilizer of $\kappa_i$ in $\Gamma$ and $\Gamma'_i$ its maximal unipotent subgroup.
Let $\zeta$ be a cusp of $\Gamma$ such that $\zeta = A^{-1}\infty$. Define $\Gamma_\zeta$ to be the stabilizer of $\zeta$ in $\Gamma$, and its maximal unipotent subgroup
\[
\Gamma_\zeta' = \{M\in \Gamma: M\zeta = \zeta, M = I \text{ or parabolic}\}, 
\]

\[
A\Gamma_\zeta A^{-1} = \Big\{\pm\begin{pmatrix} 1 &\alpha \\  0 & 1 \end{pmatrix}: \alpha\in L_i\Big\}
\]
for a lattice $L_i$ in $\C$.

Now define the weight zero Eisenstein series for $\Gamma$ at the cusp $\kappa_i$ as
\[
E^0_A(u,s) = \sum_{\sigma\in \Gamma_i'\backslash \Gamma }t(\sigma_i \sigma(u))^{1+s}. 
\]
where $t(u)$ indicates the $t$ component of $u = (z,t)$, and similarly $z(u)$. We can replace the sum over $\Gamma_\zeta'\backslash \Gamma$ with $\Gamma_\zeta\backslash \Gamma$ if we remember to divide by the finite index $[\Gamma_i': \Gamma_i]$. From \cite{MR1483315}, we have
\begin{enumerate}
\item
the function $E_A(u,s)$ is invariant under $\Gamma$, i.e., $E_A(\gamma(u),s)=E_A(u,s)$ for all $\gamma\in\Gamma$.
\item
It satisfies the differential equation $\Delta  E_A(u,s)  = (s^2-1)E_A(u,s)$ where  $z=x+iy$ and $\Delta$ is the Laplace operator on $\tH_3$ 
\[
t^2\Big(\frac{\partial^2}{\partial x^2} +\frac{\partial^2}{\partial y^2}+ \frac{\partial^2}{\partial t^2}\Big)- t \frac{\partial}{\partial t}.
\] 
\item
It also has a Fourier expansion
\[
E_i(\sigma_j^{-1}u,s) = \delta_{ij}a_0t^{1+s} + b_0 t^{1-s} + \sum_{0\neq \alpha\in L_i'}c(\alpha,s)t K_s(2\pi|\alpha|t)e^{2\pi i \langle \alpha,z \rangle}
\]
where $L_i'$ denotes the dual lattice of $L_i$ and $\langle\cdot,\cdot\rangle$ is the inner product on $\C$. The coefficients $a_0, b_0$ are independent of $i$ and $j$, while $\delta_{ij}$ is the usual Kronecker delta function.
\item
It has a meromorphic continuation to $s\in\C$, without any pole in the half plane $\text{Re}(s)>0$ except for possibly finitely many simple poles in $(0,1]$. It has  a simple pole at $s=1$ with residue equal to $|L_i|\text{vol}(\Gamma)^{-1}$.  Here $|L_i|$ is the Euclidean area of a fundamental parallelogram of $L_i$ and $\text{vol}(\Gamma)$ is the covolume of $\Gamma$.
\end{enumerate}
The last condition can be expressed as
\[
\lim_{s\to1}(s-1)E_i(u,s) =C;
\]
where $C$ is a constant independent of $i$.

Following \cite[\S3.2]{MR1483315}, define the {\it weight $0$} Eisenstein series for $\Gamma$ at the cusp $\zeta$: 

\[
E^0_A(u,s) = \sum_{M\in \Gamma_\zeta'\backslash \Gamma }t(AMu)^{1+s};
\] where $t(u):=t$ indicates the $t$ component of $u = z+jt$. It is invariant under $\Gamma$.

\subsection{Eisenstein differential forms of weight two}
Let 
\[
J(\sigma; u) = \Big(\frac{\partial}{\partial z}, \frac{\partial}{\partial t}, \frac{\partial}{\partial \bar z} \Big) z(\sigma(u)), \qquad u\in \tH_3, \sigma\in \Gamma.
\]
Define the weight 2 Eisenstein series for $\Gamma$ at the cusp $A$
\[
G_A(u,s) = \sum_{\sigma\in \Gamma_i'\backslash \Gamma }J(\sigma; u)t(\sigma_i\sigma(u))^{1+s}, 
\]
which converges absolutely for $\mathrm{Re}(s)$ large enough. If $\kappa_j$ is another cusp of $\Gamma$, the function $u\mapsto G_A(\sigma_j u,s)$ is invariant under the action of the lattice $L_j$ corresponding to $\sigma_j\Gamma'_j \sigma_j^{-1}$.  Define the Eisenstein differential to be the complex-valued differential form 
\[
\omega_A(s) = G_A(u,s)\begin{pmatrix}dz\\ dt\\ d\bar{z}\end{pmatrix}
\]
for $\mathrm{Re}(s)$ large enough.

\begin{prop}
The differential form $\omega_A(s)$ is invariant with respect to $\Gamma$, has meromorphic continuation to $s\in \C$, and is closed at $s=0$.
\end{prop}

\begin{proof}
Let $\varphi_\gamma$ be the automorphism $u\mapsto \gamma(u)$ of $\tH_3$. From the relation
\[
J(\sigma;\gamma(u))\mathscr \J(\gamma,u) = J(\sigma\gamma; u), \qquad \gamma \in \Gamma,
\]
where $\J(\gamma, u)$ denotes the Jacobian matrix

$$
\J(\gamma,u)=\frac{\partial(z(\gamma u),t(\gamma u),\bar{z}(\gamma u))}{\partial(z,t,\bar{z})}.
$$

It follows that
\[
G_A(\gamma(u),s) \J(\gamma,u) = G_A(u,s).
\]
Then
\begin{align*}
\varphi^*_\gamma \omega_A(s)  &= G_A(\gamma(u),s) \beta(\gamma(u))\\
& = G_A(u,s)\J(\gamma,u)^{-1}  \J(\gamma, u)\beta(u) \\
&= \omega_A(s),
\end{align*}
where we have used the transformation law, hence $\omega_i(s)$ is invariant under $\Gamma$.

Writing component-wise $G_A  = (G_{A}^{(1)},G_{A}^{(2)},G_{A}^{(3)})$ and $J = (J^{(1)},J^{(2)},J^{(2)})$, for each $1\le j\le 3$ we have that 
\[
G_{A}^{(j)}(u,s)= \sum_{\sigma\in \Gamma_A'\backslash\Gamma}J^{(j)}(\sigma; u) t(\sigma_A\sigma(u))^{1+s}.
\]
More explicitly, if $\sigma(u) = (\begin{smallmatrix} * & * \\ m & n \end{smallmatrix})$, then the formulas 
\[
t(\sigma(u))= (|{mz+n}|^2+|mt|^2)^{-1}t
\]
\[
J(\sigma; u) = {(|mz+n|^2+|mt|^2)^{-2}} ((\overline{mz+n})^2, 2(\overline{mz+n})\bar{m} t, -(\bar{m} t)^2),
\]
imply that 
\[
J(\sigma;u) = t(\sigma(u))^2 \Big(\frac{(\overline{mz+n})^2}{t^2}, \frac{2(\overline{mz+n})\bar{m}}{t}, -\bar{m}^2\Big)
\]
and so for $t\gg 0 $ we have 
\[
G_{A}^{(j)}(u,s) \ll \sum_{\sigma\in \Gamma_A'\backslash\Gamma}\frac{t(\sigma_A\sigma(u))^{3+s}}{t^{3-i}} = \frac{1}{t^{3-i}} E_A(u,s+2).
\]
for each $1\le j\le 3$ with the implied constant depending on $z$. Recall that the right-hand side is a scalar-valued Eisenstein series of weight zero for $\Gamma$ at cusp $\kappa_i$.\footnote{This is denoted $E_A(P,s)$ in \cite{MR1483315}.} It follows then from Corollary 3.2.4 of \cite{MR1483315} that $G_{A}^{(j)}$ has polynomial growth at all cusps $\kappa_n$ of $\Gamma$ in the sense that
\[
G_{A}^{(j)}(\sigma_n^{-1}u,s) = O(t^K)
\]
for some constant $K>0$ uniformly with respect to $z$. Moreover, by the same argument as Proposition 3.2.5 of \cite{MR1483315} we see that each $G_{i}^{(j)}(u,s)$ satisfies the differential equation
\[
\Delta G_{A}^{(j)}(u,s) = (s^2-1)G_{i}^{(j)}(u,s)
\]
and are thus real analytic functions of $u$, holomorphic in $s$ for $\mathrm{Re}(s)$ large enough. Then from the theory of Eisenstein series \cite[\S6.1]{MR1483315} it follows that each $G_{i}^{(j)}$ has meromorphic continuation to $s\in \C$ and in particular holomorphic at $s=0$, hence so is $\omega_i(s)$. Then by a well-known result of Harder the differential form $\omega_i(0)$ is closed \cite[\S4.2]{MR892187} (see also \cite[\S2]{MR870309}).
\end{proof}

\begin{cor}
\label{Explicit}
There exists a function $H_A(u,s)$ such that $dH_A(u,0) = \omega_A(0)$ for each $A$. Also, one has $\Delta H_A(u,0) = 0$.  
\end{cor}

\begin{proof}
The existence of $H_A(u,s)$ is clear. The second assertion follows from the Bessel differential equation \eqref{bessel}.
\end{proof}
Note that $0=\Delta H_A(u,0)=\star d \star d H_A(u,0)=\star d \star d \omega_A(0)$.
If we express the Fourier expansion of $H_A(u,s)$ as simply
\[
a_0t^{1+s}+ b_0t^{1-s} + \sum_{0\neq \alpha\in L' }c(\alpha,s)tK_s(2\pi|\alpha|t)e^{2\pi i \langle \alpha,z\rangle}, \quad u = (z,t) 
\]
and for $s=0$ the term $b_0t^{1-s}$ must be replaced with $b_0t\log t$ \cite[\S3]{MR1483315}. Then by the relations
\[
\frac{d}{dr} r^s K_s(r) =  -r^s K_{s-1}(r), 
\]
for Re$(s)>-\frac12$. Then if we set $c(\alpha)=c(\alpha,0)$, we can express the Fourier coefficients of $G_i(u,0)$ as simply
\begin{align*}
c_1(\alpha) &= c(\alpha)\frac{\partial}{\partial z}e^{2\pi i \langle \alpha,z\rangle} = 2\pi i c(\alpha)\frac{\partial}{\partial z}\langle \alpha,z\rangle,\\
c_2(\alpha) &= c(\alpha)\frac{\partial}{\partial t}tK_s(2\pi|\alpha|t)= 2\pi |\alpha|c(\alpha),\\
c_3(\alpha) &= c(\alpha)\frac{\partial}{\partial \bar{z}}e^{2\pi i \langle \alpha,z\rangle}= 2\pi i c(\alpha)\frac{\partial}{\partial \bar{z}}{\langle \alpha,z\rangle}.
\end{align*}
We refer to the recent paper of Miao-Nguyen-Wong~\cite{MR4642603} about explicit $H_A$.

We can compute $d \left( \star \omega_i \right) $ using the formula given in \S~\ref{Differential forms} and considering the cusp $i=\infty$.
In this case 
\[
(F_0, F_1, F_2)= (G_{A}^{(1)},G_{A}^{(2)},G_{A}^{(3)})
\]
We have
\[
d \left( \star (F \cdot \beta)\right)=i H_F \beta_0 \wedge \beta_1 \wedge \beta_2;
\]
with the function $H_F$ given by 
\begin{align*}
H_F= t\left( \frac{\partial \overline{G_{A}^{(2)}}}{\partial z}+ \frac{1}{2}\frac{\partial \overline{G_{A}^{(1)}}}{\partial t}+ \frac{\partial \overline{G_{A}^{(3)}}}{\partial z}\right). 
\end{align*}
We expect that   $\int\limits_{\Ga \backslash \tH_3} d \left( \star (F \cdot \beta)\right)$ can be computed using Rankin-Selberg ``unfolding".

 \subsection{Hida's differential forms }
Consider  Hida's  differential form at the cusp $A=\infty$ ~\cite[\S $10$]{MR1272981}:
 \begin{center}
$$
\omega_{E}:=\omega_{A(u;s)}= \sum\limits_{\gamma \in \Gamma_{\infty}\backslash \Gamma } \frac{t^s \cdot e_{1} \cdot \rho_{2} (j(\gamma,u)^{t})^{-1}}{(|cz+d|^2+|ct|^2)^s}  du;
$$
\end{center}
where 
\begin{itemize}
 \item 
 $e_{1}=(1,0,0)$ is the unit vector, 
 \item 
  $ du=(dz,-dt,-d\bar{z})$ is a differential $1$ form, 
  \item 
  $\rho_{2}$ = $Sym^2 (\mathbb{C}^2 )$ is the second  symmetric tensor
representation of the standard representation of $\SL_{2}(\mathbb{C})$ on $\mathbb{C}^2$.
\item 
For $\gamma=\left(\begin{array}{ll}
a & b \\
c & d\end{array}\right) \in \Gamma_{\xi}\backslash \Gamma$ and $u=\left(\begin{array}{cc}z & -t \\ t & \bar{z}\end{array}\right) \in \mathbb{H}_3$,
the modular factor is defined by
\[
j(\gamma; u)=\left(\begin{array}{cc}
c z+d & -c t \\
\overline{c} t & \overline{cz+d}
\end{array}\right).
\]
\end{itemize}
In particular, Hida is using $Sym^2 j(\gamma; u) $ rather than $Sym^2 j(\gamma; u)^{-1} $.
By applying the Hodge $*$ operator (cf. ~\ref{Differential forms}) to the differential form $\omega_E$, we can compute the $*\omega_E$ for Hida's Eisenstein differential form at $s=0$.

\section{Inner product formula for Bianchi modular forms}
\label{innerprod}
We assume in this section $K$ is an imaginary quadratic field of class number {\it one} that is also an Euclidean domain. Recall that that fundamental domains are described in ~\cite[Chapter 7]{MR1483315}, ~\cite[Chapter 3]{aranes2010modular}. 
\subsection{Quasi-periods}
\begin{definition}
(Quasi-periods of Bianchi modular forms)
Choose an edge of the  Fundamental domain connecting $\infty$ to two different points $P_1$ and $P_2$ of the floor of the Bianchi domain~\cite[p. 70]{aranes2010modular}. 
Now write $\{P_i, j\infty \}$ as the translate $h_i \{0, \infty\}$ with $ h\in \Ga$ as in \cite{JCthesis}. This is very crucial for our computation and the reason we assume that
$K$ is an Euclidean domain. 
 For a Bianchi modular form $F=(F_0, F_1, F_2) :\tH_3 \rightarrow \C^3$ and $u:=z+tj \in \tH_3$.
 
 For example choose the point $P_{1}={(\frac{1}{2},0,\frac{\sqrt{3}}{2})}$ for $K=\Q(i)$. We integrate over simply connected domain $\tH_3$. When integrating with respect to $dz$ and $d\bar{z}$, we consider $t$ component as constant. Similarly, when integrating with respect to $dt$, we take $z$ component as constant, and then we obtain $z=\frac{1}{2},\; \bar{z}=\frac{1}{2}\;$ and $t=\frac{\sqrt{3}}{2}$ :
\[
\pi_{F_{0}}(z):=\int_{\frac{1}{2}}^{z} {F_0(z',t)}  \frac{dz'}{t}  
 \text{\; \;for arbitrary variable point \;} (z',t) \in \mathbb{H}_3;
 \]
\[\pi_{F_{1}}(t):=\int_{\frac{\sqrt{3}}{2}}^{t} {F_1(z,t')}  \frac{dt'}{t'}   \text{\; \;for arbitrary variable point \;} (z,t') \in \mathbb{H}_3;
\] 
\[
\pi_{F_{2}}(\bar{z}):=\int_{\frac{1}{2}}^{\bar{z}} {F_2(z',t)}  \frac{d\bar{z'}}{t} \text{\; \;for arbitrary variable point \;} (z',t) \in \mathbb{H}_3;
\]
and
\[
\pi_F:= \pi_{F_{0}}(z) +   \pi_{F_{1}}(t)+\pi_{F_{2}}(\bar{z}). 
\]
Then, we have 
\[
d\pi_F= d\pi_{F_{0}}(z) +   d\pi_{F_{1}}(t)+d\pi_{F_{2}}(\bar{z}). 
\]

We have chosen $z'$ and $t'$ arbitrarily. For the sake of better notation, we replace $z'$ with $z$ and $t'$ with $t$, i.e., $u = (z, t)\in \mathbb{H}_3$;

\[
d\pi_F= {F_{0}}(u)\frac{dz}{t}  +   F_{1}(u)\frac{dt}{t} +F_{2}(u)\frac{d\bar{z}}{t}
\]

with $P_1$, $P_2$, and $P_3$ are points in the fundamental domain defined in $\S~\ref{2.1}$. 
\end{definition}
Choose functions $\pi_{F_i}$ such that we have 
\[
d ( \pi_{F_{i}})={F_i(u)}  \beta_i.
\]

We prove the following formula that works for any arbitrary subgroup $\Ga \leq G:=\SL_2(\Z[i])$ of finite index. This is a generalization of the inner product formula of Banerjee-Merel ~\cite{banerjee2022eisenstein} for imaginary quadratic fields. 

\begin{prop}
\label{pinnerprod}
Let  $F, S \in M_{2}(\Ga)$ be two Bianchi modular forms, with at least one of them being a cusp form with  $F=(F_0, F_1, F_2)$ and $S=(S_0, S_1, S_2)$.  
Consider the function
\[
H:= i F_0  \overline{S_0} +\frac{i}{2}F_1 \overline{S_1}+iF_2 \overline{S_2}. 
\]
The inner product of these two modular forms  is given by $<F, S>=I$
with
\begin{align*}
 I =&\frac{1}{12[G:\Gamma]} \sum_{g \in \Gamma \backslash G}  \int\limits_{g0}^{g\infty} \int\limits_{\partial{\cF}} H (\beta_0\wedge \beta_1 \wedge \beta_2)
\end{align*}
where $\partial{\cF}$ is the boundary of fundamental domain $\cF$ and $h \in \Ga$.
\end{prop}
\begin{proof}
 By ~\S~\ref{Bianchidefinition}, the inner product is given by 
\begin{eqnarray*}
 <F, S>
  &:= & \frac{1}{12i[G:\Gamma]}\int\limits_{\Gamma \backslash \overline{\tH}_3} F \cdot \beta  \wedge *  (S \cdot \beta)\\
    &= & \frac{1}{12i[G:\Gamma]}\sum\limits_{g \in \Gamma \backslash G} \int\limits_{\cF}  F|_g \cdot \beta  \wedge *  (S|_g \cdot \beta).
\end{eqnarray*}
In the above equations,  $\cF$ is the fundamental domain as in ~\cite[diagram 4.2]{JCthesis} with 
boundary $\partial(\cF)$. Integration over the boundary of the fundamental domain $\partial{\cF}$, as defined in ~\ref{boundary}. Now write $\{P_1, j\infty \}$ as the translate $h \{0, \infty\}$ with $ h\in \Ga$ as in \cite{aranes2010modular}. 

Consider the  integral 
\[
\int\limits_{\cF}  F|_g \cdot \beta  \wedge *  (S|_g \cdot \beta).
\]
 Choose a quasi-period  such that $ d ( \pi_{F})=F \cdot \beta$.

Note that $d(\pi_F \cdot (*(S \cdot \beta))
= d( \pi_F  ) \wedge  *  (S \cdot \beta)
+ \pi_F.d( *  (S \cdot \beta))$. We know that $d(* (S \cdot \beta))=0$ because $S$ is harmonic function. Observe that
\begin{align*}
d(\pi_F \cdot (*(S.\beta))
&=d( \pi_F  ) \wedge  *  (S \cdot \beta)
+ \pi_F \cdot d( *  (S \cdot \beta)), \\
d(\pi_F \cdot (*(S.\beta))
&= d( \pi_F  ) \wedge  *  (S \cdot \beta)
+ 0,\\
d(\pi_F \cdot (*(S.\beta))
&=F\cdot \beta \wedge  *  (S \cdot \beta).
\end{align*}

By Stokes theorem, we have 
\begin{align*}
\int\limits_{\Gamma \backslash \overline{\tH}_3} F \cdot \beta  \wedge *  (S \cdot \beta)= \int\limits_{\Gamma \backslash \overline{\tH}_3}d\left(\pi_F.*(S\cdot\beta)\right)\\
= \sum\limits_{g \in \Gamma \backslash G} \int\limits_{\cF}  d\left(\pi_{F_{|_g}}.*(S|_g \cdot\beta)\right)\\
= \sum_{g \in \Gamma \backslash G} \int\limits_{\partial(\cF)} \pi_{F_{|_g}}.*(S|_g \cdot \beta).
\end{align*}
Hence, we deduce that
\begin{center}
\begin{align*}  
\sum\limits_{g \in \Gamma \backslash G}   \left( \int\limits_{\partial (\cF)} \pi_{F_{|_g}}.(*(S|_g \cdot \beta)\right)
&= 
\sum\limits_{g \in \Gamma \backslash G} \left( \int\limits_{\partial (\cF)} \pi_{F_{|_g}}.( -\frac{1}{2}i\bar{S}_1(\beta_0\wedge\beta_2) +i\bar{S}_0(\beta_1\wedge\beta_2) +i\bar{S}_2(\beta_0\wedge\beta_1)\right)\\
&=
\sum\limits_{g \in \Gamma \backslash G} \left( \int\limits_{\partial(\cF)} -\frac{1}{2}i\pi_{F_{1{|_g}}}\bar{S}_1(\beta_0\wedge\beta_2) +i\pi_{F_{0{|_g}}}\bar{S}_0(\beta_1\wedge\beta_2) +i\pi_{F_{2{|_g}}}\bar{S}_2(\beta_0\wedge\beta_1) \right)\\
&=\sum_{g \in \Gamma \backslash G} \left( -\frac{1}{2}i\int_{\frac{\sqrt{3}}{2}}^{\infty}{F_{1{|_g}}\beta_1}\int\limits_{\partial{\cF}}\bar{S}_1(\beta_0\wedge\beta_2) +i\int\limits_{\frac{1}{2}
}^{\infty}{F_{0{|_g}}}\beta_0 \int_{\partial{\cF}}\bar{S}_0(\beta_1\wedge\beta_2) \right) \\
&+ \sum_{g \in \Gamma \backslash G} \left( i\int_{\frac{1}{2}}^{\infty}{F_{2{|_g}}\beta_2}\int_{\partial{\cF}}\bar{S}_2(\beta_0\wedge\beta_1) \right)\\
&=\sum_{g \in \Gamma \backslash G} \left( -\frac{1}{2}i\int_{0}^{\infty}{F_{1{|_{gh}}}\beta_1}\int_{\partial{\cF}}\bar{S}_1(\beta_0\wedge\beta_2) +i\int_{0}^{\infty}{F_{0{|_{gh}}}\beta_0}\int_{\partial{\cF}}\bar{S}_0(\beta_1\wedge\beta_2) \right) \\
&+ \sum\limits_{g \in \Gamma \backslash G} \left( i\int_{0}^{\infty}{F_{2{|_{gh}}}\beta_2}\int_{\partial{\cF}}\bar{S}_2(\beta_0\wedge\beta_1) \right).
\end{align*}
 \end{center}

As a result, we deduce that
\begin{center}
\begin{align*}
 I  &= \frac{1}{12[G:\Gamma]} \sum_{g \in \Gamma \backslash G} \left( -\frac{1}{2}\int_{0}^{\infty}{F_{1{|_{gh}}}\beta_1}\int_{\partial{\cF}}\bar{S}_1(\beta_0\wedge\beta_2) +\int_{0}^{\infty}{F_{0{|_{gh}}}\beta_0}\int_{\partial{\cF}}\bar{S}_0(\beta_1\wedge\beta_2) \right) \\
& +\frac{1}{12[G:\Gamma]} \sum_{g \in \Gamma \backslash G} \left( \int_{0}^{\infty}{F_{2{|_{gh}}}\beta_2}\int_{\partial{\cF}}\bar{S}_2(\beta_0\wedge\beta_1) \right)
\end{align*}
\end{center}
As $g$ varies in $\Gamma \backslash G$ so do $gh$ for a fixed $h \in G_{P_1}$.

\begin{center}
\begin{align*}
 I  &= \frac{1}{12[G:\Gamma]}\sum_{g \in \Gamma \backslash G} \left( -\frac{1}{2}\int_{g0}^{g\infty}{F_1\beta_1}\int_{\partial{\cF}}\bar{S}_1(\beta_0\wedge\beta_2) +\int_{g0}^{g\infty}{F_0\beta_0}\int_{\partial{\cF}}\bar{S}_0(\beta_1\wedge\beta_2) \right) \\
  & + \frac{1}{12[G:\Gamma]}\sum_{g \in \Gamma \backslash G} 
 \int_{g0}^{g\infty}{F_2\beta_2}\int_{\partial{\cF}}\bar{S}_2(\beta_0\wedge\beta_1) 
\end{align*}
\end{center}
By collecting the terms with respect to the volume form $\beta_0\wedge\beta_1 \wedge \beta_2$, we deduce that 
\[
I= \frac{1}{12[G:\Gamma]}\sum_{g \in \Gamma \backslash G} \int_{g0}^{g\infty} \int_{\partial{\cF}} H (\beta_0\wedge \beta_1 \wedge \beta_2)
\]
\end{proof}
 We now prove our Theorem ~\ref{Main-thm}. 
 \begin{proof}
For an Eisenstein series $E \in E_2(\Ga)$, define 
 \[
 F_E(g):=  \int\limits_{\partial{\cF}} *  (E \cdot \beta) [g]. 
\]
Using ~\cite[theorem 1]{MR743014}, consider the modular symbol 
$\sE_E \in \HH_1(X^{BB}_\Ga, \partial(X^{BB}_\Ga); \C):=\HH_1(X^{BB}_\Ga, \partial(X^{BB}_\Ga); \Q)\otimes \C$ as in the statement of the theorem:
\[
\sE_E:=\sum_{g \in \Ga \backslash G} F_E(g) \eta(g). 
\]
We compute the integral as follows:
\[
I= \frac{1}{12[G:\Gamma]}\int\limits_{\sE_E} F \cdot \beta 
 =  \frac{1}{12[G:\Gamma]} \sum\limits_{g \in \Ga \backslash G}F_E(g) \int\limits_{\eta(g)} F \cdot \beta. 
 \]
 
Consider two weight two Bianchi modular forms $F=(F_0, F_1, F_2)$ and $E=(E_0, E_1, E_2)$ in $M_2(\Ga)$.
Consider the function 
\[
H:= i F_0  \overline{E_0} +\frac{i}{2}F_1 \overline{E_1}+iF_2 \overline{E_2};
\]
as in  Proposition~\ref{pinnerprod}.

By Proposition~\ref{pinnerprod}, we have 
\begin{eqnarray*}
<F, E>  &= & \frac{1}{12[G:\Gamma]}\sum_{g \in \Ga \backslash G} \int_{\cF}  H|_g \beta_0 \wedge \beta_1\wedge \beta_2. 
 \end{eqnarray*}

We then have 
\begin{align*}
\int\limits_{\sE_E} F \cdot \beta
=&
\sum\limits_{g \in \Ga \backslash G} F_E(g) \int\limits_{\eta(g)} F \cdot \beta\\
=&
\sum_{g \in \Ga \backslash G}  \int\limits_{\eta(g)} \int_{\partial{\cF}} H|_g \beta_0 \wedge \beta_1 \wedge \beta_2.\\
\end{align*}
By Proposition~\ref{pinnerprod}, the right-hand side  is equal to 
\[
-12 i [G:\Ga] \langle F, E \rangle = 0. 
\]
Since, by definition,  Eisenstein modular forms are defined to be the complement of cusp forms. 
\end{proof}
We now compute the integral a bit explicitly under certain assumptions. Recall that by Corollary~\ref{Explicit}; we have: 
\[
dH_E=\omega_E;
\]
for a function (degree zero differential form) $H_E$.
\begin{prop}
\label{nonzero}
Assume that $\delta(H_E \omega) $ is a non-vanishing function on the Riemannian two manifold $\partial{\cF}$ then we have $F_E(I) \neq 0$. 
\end{prop}
\begin{proof}
We then have $\star \omega_E=\star dH_E$. Let $\omega$ be the non-vanishing top form on the Riemannian manifold $Y_{\Ga}$. By our assumption on $H_E$, we have  $H_E \omega$ is a top form on the Riemannian $3$ manifold.
Observe that 
\[
\delta(H_E \omega)=-\star d \star (H_E \omega)=-\star d H_E =-\star \omega_E
\]
By assumption, we have  $\delta(H_E \omega)$ is a volume form on $2$ dimensional submanifold $\partial{\cF}$.
Hence, we have 
\[
\int\limits_{\partial{\cF}} \star \omega_E=-\int\limits_{\partial{\cF}} \delta(H_E \omega) \neq 0.
\]

\end{proof}
Note that the assumption is reasonable since $\eta(z) \neq 0$ for all $z$ in the upper half plane. 
\subsection{Eisenstein elements  for $\Ga_0(p)$}

Consider the subgroup $\Ga=\Ga_0(p)$ for a prime $p$ of $\sO_K$ that is inert. In this case, there are only two cusps $[0]$ and $[\infty]$ similar to the classical 
case and denote the corresponding modular curve by $Y_0(p)$. There are two Eisenstein series.  
Define
\[
\partial{\cF}(p):=\int\limits_{\partial{\cF}}  \star d H_{E}.
\]

We compute the boundary of the Eisenstein element explicitly in this case. 
\begin{prop}
Consider the congruence subgroup $\Ga_0(p)$ with $p$ as above.  Under the assumption on $H_E$ as above. The boundary of the Eisenstein element is non-zero and determined by $\partial{\cF}(p)$. 
\end{prop}
\begin{proof}
\begin{align*}
\delta(\sE_{E})&=& \sum_{g \in \Ga \backslash \SL_2(\sO_K)} \int_{\partial{\cF}} \star \omega_{E}[g] ([g0]-[g\infty])\\ 
&=& \left(\sum_{g \in \Ga \backslash \SL_2(\sO_K)} \int_{\partial{\cF}} \star \omega_{E}[g]\right) ([0]-[\infty])\\
&=& \left(\int\limits_{Y_0(p)} \star \omega_{E}\right) ([0]-[\infty])\\ 
&=& \left(\int\limits_{Y_0(p)} \star d H_{E}\right) ([0]-[\infty])\\ 
\end{align*}
for a function $H_E$ as in Corollary~\ref{Explicit}.
In particular, it shows that $\delta(\sE_{E}) \neq 0$ if $\int\limits_{Y_0(p)}  \star d H_{E}=[\SL_2(\sO_K):\Ga_0(p)]\partial{\cF}(p)\neq 0$ by Proposition~\ref{nonzero}. 
\end{proof}
\subsection{Different pairings of homology and cohomology groups}
 Let $\Ga$ be a torsion-free 
subgroup like $\Ga_1(N)$ with $N$ large enough. Consider  the compactification $X \in \{X^{BB}_{\Ga}, X^{BS}_{\Ga} \}$ of the Riemmanian $3$ manifold $Y_{\Ga}$. 
We have isomorphisms:
\[
S_2(\Ga) \simeq \HH^1(X; \C) \simeq \HH_1(X; \C)
\]
The first isomorphism is obtained by the map $F \rightarrow F \cdot \beta$ and then using classical deRham's theorem. 
The second isomorphism is given by the duality induced by the evaluation pairing $(\omega, \gamma) \rightarrow \int\limits_{\gamma} \omega$.

We also have a pairing~\cite[p. 474]{MR1272981}, 
\[
<,> \HH_c^1(Y_{\Ga};\C) \times \HH^2(Y_{\Ga}; \C) \rightarrow \C
\]

We also have the following isomorphism (cf.~\cite[p. 288]{MR4569265}):

$$
\begin{array}{l}
\HH_{1}(X, \partial(X);\C) \simeq \HH_c^1(Y_{\Ga};\C) \simeq \HH^{2}(Y_{\Gamma};\C).
\end{array}
$$

With the identification $\HH^1_{dR}(Y_{\Ga}) \simeq  \HH^1(Y_{\Ga};\C)$, we can make this pairing explicit:
\[
 \HH^1_c(Y_{\Ga};\C) \times \HH_1(X, \partial(X);\C) \rightarrow \C.
\]
The map is given again by the same evaluation pairing $(\omega, \gamma) \rightarrow \int\limits_{\gamma} \omega$.

Recall the formulation of Poincar\'e-Lefeschtz duality \cite[p. 53]{MR1288523}.  In this case, the intersection 
pairing will be given by  
\[
\circ: \HH_{1}(X, \partial (X);\Z) \times \HH_2(Y_{\Gamma};\Z) \rightarrow \Z. 
\]

The Eisenstein element $\sE_E$ corresponding to the weight two Eisenstein modular form $E$   is  the  unique element such that $\sE_E \circ c=\int_c \star (E \cdot \beta)$ 
for all $c \in \HH_2(Y_{\Gamma};\Z)$. We expect these two definitions of Eisenstein elements to coincide.

Using the Eichler-Shimura-Harder correspondence, we know that $S_2(\Gamma)$ is isomorphism $\HH^1_{cusp}(X;\mathbb{C})$ via the map $F$ to $F \cdot \beta$. We know $\HH^1(X;\C) = \HH^1_{cusp}(X;\C) \oplus \HH^1_{Eis}(X;\C),$ and if $E$ is an Eisenstein series, then $E \cdot \beta$ is a differential form belonging to $\HH^1_{Eis}(X; \C).$

\subsection{Eisenstein elements are retractions}
Since we are using Bailey-Borel-Satake compactification for homology groups, we have a short exact sequence:
\[
0=\HH_1(\partial(X^{BB}_\Gamma);\Z)  \rightarrow  \HH_1(X^{BB}_\Gamma;\Z) \rightarrow 
\HH_1(X^{BB}_\Gamma, \partial(X^{BB}_\Gamma);\Z) \xrightarrow[]{\delta} \Z[ \partial(X^{BB}_{\Gamma})]
  \rightarrow \Z \rightarrow 0,
\]
where the map $\HH_1(X^{BB}_\Gamma, \partial(X^{BB}_\Gamma);\Z) \xrightarrow[]{\delta} \Z[ \partial(X^{BB}_{\Gamma})]$ is obtained from the boundary map 
\[
\delta:\HH_1(X^{BB}_\Gamma, \partial(X^{BB}_\Gamma);\Z) \rightarrow \HH_0( \partial(X^{BB}_\Gamma);\Z).
\]
The exact sequence above splits over the field $\C$. We have a retraction map: 
\[
R: \HH_1(X^{BB}_\Ga, \partial(X^{BB}_\Ga); \C) \rightarrow \HH_1(X^{BB}_\Ga; \C)=\Hom_\C(\HH^1(X^{BB}_{\Ga}; \C),\C)
\] 
given by $R(c)(\omega)=\int\limits_c \omega$. 

Note that this is a section of the inclusion map. Hence if we tensor with $\C$, we have 
$\delta(x)=0$ if and only if $R(x)=x$. That is same as $x \in \HH_1(X^{BB}_\Gamma;\C) $. 

By definition of Eisenstein element, we have $R(\sE_E)=0$.  Note that $\sE_E \neq 0$ as $F_E(I) \neq 0$ for $I$ the identity element by Proposition~\ref{nonzero} under certain assumption. Hence, we deduce that 
$\delta(\sE_E) \neq 0$ with the same assumption. 

\section{Eisenstein Cohomology groups for $\Ga_1(N)$}
\subsection{Dimension of Eisenstein cohomology}
Recall the definition of Eisenstein cohomology groups following  Seng\"un-T\"urkelli ~\cite{MR3464620}. 
We compute the cohomology groups of  Borel-Serre compactification of $Y_{\Ga}$. This is a compactifcation obtained by adding a $2$ torus to every cusp (except for $\mathbb{Q}( i )$ or $\mathbb{Q}(\sqrt{-3})$).
The Borel-Serre compactification ~\cite[appendix]{MR272790} 
$X^{BS}_{\Gamma}$ of $Y_{\Gamma}$ is a compact three manifold with boundary $\partial (X^{BS}_{\Gamma})$ whose interior is homeomorphic to $Y_{\Gamma}$.

Given $k \geqslant 0$, let the space of homogeneous polynomials of degree $n$ on the variables $x, y$ with complex coefficients be denoted by $\C[x, y]_{k}$. The modular group $\SL_2(\C)$ acts on this space in the obvious way permitted by the two variables. Consider the $\mathrm{SL}_{2}(\mathbb{C})$-module
\[
M_{k}:=\mathbb{C}[x, y]_{k} \otimes_{\mathbb{C}} \overline{\mathbb{C}[x, y]_{k}}
\]
where the overline on the second factor indicates that the action on this factor is twisted with complex conjugation. 
Considered as a module, $M_k$ gives rise to a locally constant sheaf $\mathcal{M}_{k}$ on $Y_{\Gamma}$ whose stalks are isomorphic to $M_k$.

In Sengün-Türkelli~\cite{MR3464620}, they use the notation $E_{k,k}$ instead of $M_k$ and for the sheaf $\sE_k$ instead of $\mathcal{M}_k$.

Considered as a module, $M_k$ gives rise to a locally constant sheaf $\mathcal{M}_{k}$ on $Y_{\Gamma}$ whose stalks are isomorphic to $M_k$. Consider the long exact sequence
\[
\ldots \rightarrow \HH_{c}^{i}\left(Y_{\Gamma} ; \mathcal{M}_{k}\right) \xrightarrow{incl^i} \HH^{i}\left(X^{BS}_{\Gamma} ; \overline{\mathcal{M}}_{k}\right) \xrightarrow{res^i} \HH^{i}\left(\partial(X^{BS}_{\Gamma}) ; \overline{\mathcal{M}}_{k}\right) \rightarrow \ldots
\]
where $\HH_{c}^{i}$ denotes compactly supported cohomology and $\overline{\mathcal{M}}_{k}$ is a natural extension of $\mathcal{M}_{k}$ to $X_{\Gamma}$.
Image of the compactly supported cohomology by the inclusion map $incl$ inside the whole cohomology group is called the cuspial cohomology group. 
Note that this is the same as the kernel of the restriction map $res$.

By the Eichler-Shimura-Harder isomorphism theorem, This consists of all cohomological cuspidal automorphic representations (cusp forms) \cite[p. 409]{MR4356858}. 
Note that the long exact sequence associated with the pair $\left(X^{BS}_{\Gamma}, \partial( X^{BS}_{\Gamma})\right)$ as above is compatible with the action of the involution $\tau$ from Section~\ref{Lefschetz numbers}.

\begin{definition}
\label{Eisensteincuspidal}
(Eisenstein cohomology groups)
The kernel of the restriction map is known as cuspidal cohomology $\HH_{cusp}^i$. The complement of cuspidal cohomology within $\HH^i$ is Eisenstein cohomology $\HH_{Eis}^{i}$. This is isomorphic to the image of the restriction map inside the cohomology of the boundary.
\end{definition}

The decomposition $\HH^{i}=\HH_{cusp}^{i} \oplus \HH_{Eis}^{i}$ respects the Hecke action.
By construction, the embedding $Y_{\Gamma} \hookrightarrow X_{\Gamma}$ is a homotopy invariance.  We have the following isomorphisms
\begin{align*}
\label{Groupcoh}
\HH^{i}\left(X^{BS}_{\Gamma} ; \overline{\mathcal{M}}_{k}\right) \simeq \HH^{i}\left(Y_{\Gamma} ; \mathcal{M}_{k}\right) \simeq \HH^{i}\left(\Gamma ; M_{k}\right).
\end{align*}
Note that the last one is the group cohomology. 
Via the above isomorphisms, we define the cuspidal (respectively Eisenstein parts) of the cohomology groups $\HH^{i}\left(\Gamma; M_{k}\right)$. Cuspidal parts (respectively Eisenstein part) consisting of elements that are zero (non-zero)  on parabolic elements of $\Gamma$. In other words, these are in kernel (not in kernel)
of $\mathrm{res}^i$.
We denote the  Eisenstein part of the group cohomology by $\HH_{Eis}^{1}(\Gamma; M_{k})$. We study this in our paper. 

The boundary $\partial(X^{BS}_{\Gamma})$ is a disjoint union of $2$-tori, each corresponds to a cusp of $Y_{\Gamma}$ (except for $\mathbb{Q}( i )$ or $\mathbb{Q}(\sqrt{-3})$). The set of cusps of $\Gamma$ can be identified with the orbit space $\Gamma \backslash \mathbb{P}^{1}(K)$.
 It is well known that when $\Gamma$ is the full Bianchi group, the number of cusps is $h(K)$, the class number of $K$. Recall that we assumed $h(K)=1$. 
Let $\mathcal{C}_{\Gamma}$ be the set of cusps of the congruence subgroups of the form $\Gamma$.
We are grateful to Professor Sengün for the outline of the proof (see also~\cite{MR3091734} for full subgroups). 
\begin{prop} 
~\cite[Proposition 4.1. p. 247]{MR3464620}
Let $\Gamma$ be a congruence subgroup of a Bianchi group.  Then
$$
\begin{gathered}
\operatorname{dim} \HH^{0}\left(\partial (X^{BS}_{\Gamma}); \overline{\mathcal{M}}_{k}\right)=\operatorname{dim} \HH^{2}\left(\partial(X^{BS}_{\Gamma}); \overline{\mathcal{M}}_{k}\right)=\# \mathcal{C}_{\Gamma} \\
\operatorname{dim} \HH^{1}\left(\partial(X^{BS}_{\Gamma}); \overline{\mathcal{M}}_{k}\right)=2 \cdot \# \mathcal{C}_{\Gamma}.
\end{gathered}
$$
\end{prop}

 It follows from algebraic topology that for $k>0$, the image of the restriction map.
\[
\HH^{i}\left(X^{BS}_{\Gamma}; \overline{\mathcal{M}}_{k}\right) \rightarrow \HH^{i}\left(\partial 
(X^{BS}_{\Gamma}); \overline{\mathcal{M}}_{k}\right)
\]
is onto when $i=2$ and its image has half the rank of the target space when $i=1$. Hence, we have the following result. 

\begin{prop} \label{prop8}
Let $\Gamma$ be a congruence subgroup as above. Then

if $k=0$,
\begin{enumerate}
\item
$\operatorname{dim} \HH_{Eis}^{0}\left(X^{BS}_\Gamma; \C\right)=1$
\item
$ \operatorname{dim} \HH_{Eis}^{1}\left(X^{BS}_\Gamma;\C\right)=\# \mathcal{C}_{\Gamma}$ 
\item
$ \operatorname{dim} \HH_{Eis}^{2}\left(X^{BS}_\Gamma;\C\right)=\# \mathcal{C}_{\Gamma}-1$.
\end{enumerate}

If $k>0$,
\begin{enumerate}
\item
$\operatorname{dim} \HH_{Eis}^{0}\left(X^{BS}_\Gamma; \overline{\mathcal{M}}_{k}\right)=0$
\item
$ \operatorname{dim} \HH_{Eis}^{i}\left(X^{BS}_\Gamma;\overline{\mathcal{M}}_{k}\right)=\# \mathcal{C}_{\Gamma}$ for $i=1,2$.
\end{enumerate}

\end{prop} 
\begin{proof}
Consider the long exact sequence associated with $(X^{BS}_{\Gamma},\partial( X^{BS}_\Gamma)$:
\begin{align*}
\ldots \rightarrow \HH^1_c(Y_{\Gamma}; {\mathcal{M}}_{k}) \xrightarrow{incl^1} \HH^1(X^{BS}_{\Gamma}; \overline{\mathcal{M}}_{k}) \xrightarrow{res^1} \HH^1(\partial (X^{BS}_{\Gamma}); \overline{\mathcal{M}}_{k}) \xrightarrow{}
\HH^2_c(Y_{\Gamma}; \mathcal{M}_{k})\\ \xrightarrow{incl^2} \HH^2(X^{BS}_{\Gamma}; \overline{\mathcal{M}}_{k}) \xrightarrow{res^2} \HH^2(\partial(X^{BS}_{\Gamma});\overline{\mathcal{M}}_{k}) 
\xrightarrow{} \HH^3_c(Y_{\Gamma}; \mathcal{M}_{k}) \xrightarrow{} \HH^3(X^{BS}_{\Gamma}; \overline{\mathcal{M}}_{k})=0.
\end{align*}
The subscript $"_c"$ denotes compactly supported cohomology. Using the Poincaré duality:
\[
\HH_c^i(Y_{\Ga};{\mathcal{M}}_{k}) \times \HH^{3-i}(X^{BS}_{\Ga}; \overline{\mathcal{M}}_{k}) \rightarrow \C,  \; \; \; \; \;   \text{for}  \;\text{i} \in \{0,1,2,3\},
\] 
we have $\HH^3_c(Y_\Gamma; \mathcal{M}_{k})$ is isomorphic to $\HH^0(X^{BS}_\Ga;\overline{\mathcal{M}}_{k})$.

Consider the map:
\begin{center}
$res^2:\HH^2(X^{BS}_{\Gamma};\overline{\mathcal{M}}_{k}) \rightarrow \HH^2(\partial (X^{BS}_{\Gamma});\overline{\mathcal{M}}_{k})$
\end{center}

If $k=0$ then $\mathcal{M}_{k}=\C$ is a constant sheaf, we get
\begin{center}
$ \HH^0(X^{BS}_{\Gamma};\C) = \C$ 
\end{center}
Using Proposition~\ref{prop8},  we deduce that $\dim_{\C}\HH^0_{Eis}(X^{BS}_{\Gamma};\C)=1$. 
By Poincar\'e duality, we have  $\HH^3_c(Y_{\Gamma}; \C)\simeq \HH^0(X^{BS}_{\Gamma};\C),$
 this implies $\HH^3_c(Y_{\Gamma}; \C) \simeq \C$.

When $\mathcal{M}_{k} \simeq \C$ is a trivial sheaf, the image of the restriction map $res^2$
has one dimensional cokernel  by Harder ~\cite[Proposition 4.7.1]{MR633658}. 

Hence, we have $\dim_{\C}\HH^2_{Eis}(X^{BS}_{\Gamma};\C)$ is $\# \mathcal{C}_\Ga-1$.  Consider the long exact sequence associated with $(X^{BS}_{\Gamma},\partial( X^{BS}_\Gamma)$:

\begin{align*}
0 \rightarrow \HH^0(X^{BS}_{\Gamma}; \C) \xrightarrow{res^0} \HH^0(\partial (X^{BS}_{\Gamma}); \C) \xrightarrow{\partial^0}  \HH^1_c(Y_{\Gamma}; \C) \xrightarrow{incl^1} \HH^1(X^{BS}_{\Gamma}; \C)  \xrightarrow{res^1} \HH^1(\partial (X^{BS}_{\Gamma}); \C) \\ \xrightarrow{\partial^1}
\HH^2_c(Y_{\Gamma}; \C) \xrightarrow{incl^2} \HH^2(X^{BS}_{\Gamma}; \C) \xrightarrow{res^2} \HH^2(\partial(X^{BS}_{\Gamma});\C) 
\xrightarrow{\partial^2} \HH^3_c(Y_{\Gamma}; \C) \xrightarrow{incl^3} \HH^3(X^{BS}_{\Gamma}; \C)=0.
\end{align*}

By successive use of the rank-nullity theorem, we get the following:
\begin{enumerate}
\item $\dim_{\C} \HH^0(\partial(X^{BS}_{\Gamma}); \C)= \dim \text{Im } res^0 + \dim \text{Im } \partial^0 \implies \#\mathcal{C}_\Ga= 1 + \dim \text{Im } \partial^0$ 
\item $\dim_{\C} \HH^1_c( Y_\Ga; \C)= \dim \text{Im } \partial^0 + \dim \text{Im } incl^1 $
\item $\dim_{\C} \HH^1( X^{BS}_{\Gamma}; \C)= \dim \text{Im } incl^1 +\dim \text{Im } res^1 $
\item $\dim_{\C} \HH^1(\partial(X^{BS}_{\Gamma}); \C)= \dim \text{Im } res^1 + \dim \text{Im } \partial^1 \implies
2 \cdot \#\mathcal{C}_\Ga=  \dim \text{Im } res^1 + \dim \text{Im } \partial^1$
\item $\dim_{\C} \HH^2_c( Y_\Ga; \C)= \dim \text{Im } \partial^1 + \dim \text{Im } incl^2 $
\item $\dim_{\C} \HH^2( X^{BS}_{\Gamma}; \C)= \dim \text{Im } incl^2 +\dim \text{Im } res^2$
\item $\dim_{\C} \HH^2(\partial(X^{BS}_{\Gamma}); \C)= \dim \text{Im } res^2 + \dim \text{Im } \partial^2 \implies
 \#\mathcal{C}_\Ga=  \dim \text{Im } res^2 + \dim \text{Im } \partial^3$\\
 $ \implies \# \mathcal{C}_\Ga = \dim \text{Im } res^2 + 1$  \; \;  ( because $\HH^3_c(Y_\Ga;\C) = \C)$ \\
 $\implies \dim \text{Im } res^2 = \# \mathcal{C}_\Ga - 1.$ 
\end{enumerate}

Using Poincar\'e duality;
\begin{center}
    $\HH^1_c(Y_\Ga;\C) \iso \HH^2(X^{BS}_\Ga; \C)$\\
    $\dim \text{Im } \partial^0 + \dim \text{Im } incl^1 = \dim \text{Im } incl^2 +\dim \text{Im } res^2 $\\
 $  \# \mathcal{C}_\Ga - 1 + \dim \text{Im } incl^1 = \dim \text{Im } incl^2 + \# \mathcal{C}_\Ga - 1 $\\
 $\implies  \dim \text{Im } incl^1 = \dim \text{Im } incl^2 ;$
\end{center}
 and 
 \begin{center}
      $\HH^2_c(Y_\Ga;\C) \iso \HH^1(X^{BS}_\Ga; \C);$\\
      $\dim \text{Im } \partial^1 + \dim \text{Im } incl^2 =\dim \text{Im } incl^1 +\dim \text{Im } res^1 ;$\\
    $  \implies  \dim \text{Im } \partial^1 = \dim \text{Im } res^1. $
 \end{center}
 We know that  $\dim \HH^1(\partial(X^{BS}_{\Gamma}); \C)= \dim \text{Im } res^1 + \dim \text{Im } \partial^1$
\[ \implies
2 \cdot \#\mathcal{C}_\Ga=  \dim \text{Im } res^1 + \dim \text{Im } \partial^1\]
\[ \implies
2 \cdot \#\mathcal{C}_\Ga=  \dim \text{Im } res^1 + \dim \text{Im } res^1
\]
\[ \implies
2 \cdot \#\mathcal{C}_\Ga= 2\cdot \dim \text{Im } res^1 
\]
\[\implies \dim \text{Im } res^1 = \#\mathcal{C}_\Ga \] 
\[\implies \dim_{\C}  \HH^1_{Eis}(X^{BS}_{\Gamma}; \C) = \# \mathcal{C}_\Ga.\]

If $k> 0$ then $\mathcal{M}_{k}$ is a locally constant sheaf, we get
\begin{center}
 $\HH^3_c(Y_{\Gamma}; {\mathcal{M}}_{k}) = \HH^0(X^{BS}_{\Gamma};\overline{\mathcal{M}}_{k}), \; \;
 \HH^0(X^{BS}_{\Gamma};\overline{\mathcal{M}}_{k}) \simeq \HH^0(\Ga;M_{k})$ 
\end{center}
we know
$\HH^0(\Ga;M_{k})=0$, this implies \; $\HH^0(X^{BS}_{\Gamma};\overline{\mathcal{M}}_{k})=0$ and $ \HH^3_c(Y_{\Gamma}; {\mathcal{M}}_{k}) = 0.$

For $k>0$, the restriction map is surjective.
 we analyze the image of the restriction maps following Harder ~\cite[Proposition 4.7.1]{MR633658}.  From loc. cit., the image of the restriction map
\[
\HH^{i}\left(X^{BS}_{\Gamma}; \overline{\mathcal{M}}_{k}\right) \rightarrow \HH^{i}\left(\partial (X^{BS}_{\Gamma}); \overline{\mathcal{M}}_{k}\right)
\]
is onto when $i=2$ and its image has half the rank of the target space when $i=1$. So we have
\begin{center}
   $ \operatorname{dim} \HH_{Eis}^{1}\left(X^{BS}_\Gamma;\overline{\mathcal{M}}_{k}\right)=\frac{1}{2} \operatorname{dim} \HH^{1}\left(\partial X^{BS}_{\Gamma}; \overline{\mathcal{M}}_{k}\right)$ \\
    $ \operatorname{dim} \HH_{Eis}^{2}\left(X^{BS}_\Gamma;\overline{\mathcal{M}}_{k}\right)= \operatorname{dim} \HH^{2}\left(\partial X^{BS}_{\Gamma}; \overline{\mathcal{M}}_{k}\right)$ 
   
\end{center}
 \begin{center}
  This implies  $\operatorname{dim} \HH_{Eis}^{i}\left(X^{BS}_\Gamma;\overline{\mathcal{M}}_{k}\right)=\# \mathcal{C}_{\Gamma} \quad$ for \;$i=1,2$.
  \end{center}
 
\end{proof}
This just shows that the dimension of the Eisenstein cohomology is determined by the number of cusps. 

\subsection{Computation of number of cusps for $\Ga_1(N)$}
By \cite[p. 165]{MR799662}, the set of cusps for $\SL_2(\sO_K)$ can be identified 
with the class group of $K$. Recall that we assumed that the class number of the imaginary number field $K$ is one. 

For the rest of the paper, we are interested in congruence subgroups of the form $\Ga=\Ga_1(\mfa)$ and denote the corresponding cuspidal group by 
$\mathcal{C}_N$ for an ideal $\mfa=(N)$. 
Let $P:=\left\{\pm\left[\begin{array}{cc}1 & j \\ 0 & 1\end{array}\right]: j \in \cO_K \right\}$ be the parabolic subgroup of $\mathrm{SL}_2(\cO_K)$. Note that 
$\bar{P}=\left\{\pm\left[\begin{array}{ll}1 & j \\ 0 & 1\end{array}\right]: j \in \cO_K / N \cO_K \right\}$ and $\Bar{P}_{+}=\left\{\left[\begin{array}{ll}1 & j \\ 0 & 1\end{array}\right]: j \in \cO_K \right\}$ is the "positive" half of $P$. Here,  the overbar signifies reduction modulo $(N)$.
For $N\in \N$, let $\mathcal{C}_{N}$ be the set of cusps of congruence subgroups of the form $\Gamma_1(N )$. 
\begin{prop}
\label{Cuspdescription}
\[ 
\#\mathcal{C}_N = | \bar{P}_{+} \backslash \mathrm{SL}_2(\cO_K / N \cO_K) / \bar{P} |
\]
\end{prop}
\begin{proof}

The group $\SL_2(\sO_K)$
acts transitively on the set of cusps. 
Recall that  $P=\mathrm{SL}_2(\cO_K)_{\infty}$ is the subgroup of $\mathrm{SL}_2(\cO_K)$ fixing the cusp  $\infty$. The map
$$
\Gamma_{1}(N) \backslash \mathrm{SL}_2(\cO_K) / P \longrightarrow\{\text { cusps of } \Gamma_{1}(N)\}
$$
given by
\[
\Gamma_{1}(N) \alpha P \mapsto \Gamma_{1}(N) \alpha(\infty)
\]
is a bijection ~\cite{MR2112196} (the proof works for imaginary quadratic fields also since the class number is one).  We know  that $\mathrm{SL}_2(\cO_K) / P$ identifies with $ K \cup\{\infty\}$, so that the double coset space $\Gamma_{1}(N) \backslash \mathrm{SL}_2(\cO_K) / P$ 
get identified with the cusps $\Gamma_{1}(N) \backslash({K} \cup\{\infty\})$.

The double coset space $\Gamma(N) \backslash \SL_2(\cO_K) / P$ is naturally viewed as $\mathrm{SL}_2(\cO_K / N \cO_K) / \bar{P}$ where $\bar{P}$ denote the projected image of $P$ in $\mathrm{SL}_2(\cO_K / N \cO_K)$, that is 
\[
\bar{P}=\left\{\pm\left[\begin{array}{ll}1 & j \\ 0 & 1\end{array}\right]: j \in \cO_K / N \cO_K \right\}.
\]
We have a decomposition $ \Gamma_{1}(N)$ = $\bigcup_j\left[\begin{array}{ll}1 & j \\ 0 & 1\end{array}\right] \Gamma(N)$,
the double coset space is naturally viewed as $\bar{P}_{+} \backslash \mathrm{SL}_2(\cO_K / N \cO_K) / \bar{P}$ where $\Bar{P}_{+}=\left\{\left[\begin{array}{ll}1 & j \\ 0 & 1\end{array}\right]: j \in \cO_K \right\}$ is the "positive" half of $P$ and again the overbar signifies reduction modulo $(N)$. 

We deduce that
\[
 | \bar{P}_{+} \backslash \mathrm{SL}_2(\cO_K / N \cO_K) / \bar{P} | = \# \mathcal{C}_{N}. 
 \]

\end{proof}

The cusps of $\Gamma_{1}\left(p^n\right)$ are in bijection with the sets 
$\{\pm(\bar{x}, \bar{y})\} \subset\left(\mathcal{O  } /\left(p^n\right)\right)^2$
such that the order of $(\bar{x}, \bar{y})$ is $p^n$.   The bijection is defined via the map $\frac{x}{y} \mapsto(y,-x)$ with $x\text{ modulo} \ gcd(y,N)$. 

\begin{cor}

$\#\mathcal{C}_{p} = p^2 -1$
\end{cor}
\begin{proof}
We know $\cO_K / p \cO_K$ is isomorphic to $ \mathbb{Z} /p \mathbb{Z} \oplus \mathbb{Z} /p \mathbb{Z} $ and number of non-zero choices of $y  \in \mathbb{Z} /p \mathbb{Z} \oplus \mathbb{Z} /p \mathbb{Z}$ is $p^2-1$ then $gcd(y,p)=1$.
Possibilities of $x$ is $x \ \text{modulo} \ gcd(y,p)$ and $gcd(y,p)=1$ then $x \ \text{modulo} \ 1$. The number of choices of $x$ is $1$, and the number of choices for $(x,y)$ is $p^2-1$.
\end{proof}
\subsection{Szech cocycles and some expectations}
\label{Szech}
The classical Dedekind sums arise from the homomorphism
\[
\phi: \Gamma \to \C
\]
for $\Gamma = SL_2(\Z)$, given by periods of Eisenstein series
\[
\phi=\phi_E: \gamma \mapsto \int_{[\gamma]}\omega=   \int_{z_0}^{\gamma z_0} E(z) dz 
\]
where $E(z)$ is the unique Eisenstein series (non-holomorphic modular form) of weight two on $\Gamma$.

Mazur's work uses the analogous Rademacher homomorphism for $\Gamma = \Gamma_0(N)$ and $\Ga(N)$, while Merel studies the case $\Gamma = \Gamma_0(N) $, where in each case $E(z)$ is replaced by the holomorphic Eisenstein series of weight two on $\Gamma(N)$ and $\Gamma_0(N)$ respectively. These periods are described by Dedekind sums and, therefore, are seen to be integral.  The work of Merel and the first author shows that the Eisenstein cycle can be written as a linear combination of Manin symbols $\xi$, 
\[
\sE_E = \sum_{x\in \mathbb P^1(\Z/N\Z)} F_E(x) \xi(x)= \sum_{x\in \Ga_0(N) \backslash \SL_2(\Z)} F_E(x) \xi(x).
\]
where the coefficients are shown to be a scalar multiple of $\phi_E$, and therefore can be given in terms of Dedekind sums, thus Bernoulli numbers and special values of $L$ functions. 

Over an imaginary quadratic field \cite{MR870309} showed that for $\Gamma = SL_2(\mathcal O_K)$, the period is again described by elliptic Dedekind sums \cite{scz} introduced by Sczech as discussed above. In the case of congruence subgroups of the form $\Ga_1(N)$, however, the analogue of elliptic Dedekind sums have not yet been defined to the best of our knowledge. 
 The explicit Sczech $1$-cocycles~\cite{MR850124} (see also Sengün-Türkelli~\cite[\S 4.2.3, p. 252]{MR3464620}) can be used to produce a basis for group cohomology $\HH_{Eis}^{1}(\Gamma; \mathbb{C})$ for some subgroup $\Ga \leq \SL_2(\sO_K)$. 

 Szech cocycles are expected to be computed using the cocycle $\phi_E$ associated to the Bianchi Eisenstein
 modular form $E$ of weight two.
 For a lattice $L=\sO_K$, consider the subgroup  
 \[
 \Gamma(u,v) = \{ A \in \SL_{2}(\cO_K) | (u,v)A=(u,v) \}  \text{ for  } (u,v) \in (\mathbb{C}/L)^2\}.
 \]
 Note that $\Ga_1(N)$ is a subgroup of $ \SL_{2}(\cO_K)$
of the form $\Gamma(u,v)$ for $(u,v) = (0,1) \in\left(\frac{1}{N} \cO_K / \cO_K\right)^{2}$ .  
On the other hand,  it is easy to see that the congruence subgroup $\Gamma_{0}(\cP)$ is not a subset of the form  $\Gamma(u,v)$ for all non-zero $(u,v) \in (\C/L)^2$ for a prime ideal $\cP$ be a prime ideal of $\cO_K$. 

We are interested in subgroups of the form $\Ga_1(N)$ for the rest of the paper in the setting of  ~\cite[Theorem 9, p. 101]{MR850124}.

Sczech shows that if the congruence subgroup is of the form $\Ga(u,v)$ for a {\it fixed} $(u, v) \in\left(\frac{1}{N} \cO_K / \cO_K\right)^{2}$ like $\Ga_1(N)$, we can define the collection of homomorphisms $\Psi(u, v)$ with $(u, v) \in\left(\frac{1}{N} \cO_K / \cO_K\right)^{2}$ that live in the Eisenstein part of the cohomology and that the number of linearly independent such homomorphisms is equal to the number of cusps of $\Gamma$. Thus they generate $\HH_{Eis}^{1}(\Gamma_1(N); \mathbb{C})$.
Note that the above Eisenstein series associated with different cusps are linearly independent because they are non-vanishing only at their associated cusp. This implies that the cohomology classes of Sczech cocycles, which are associated with the cusps of $\Gamma_{1}(N)$ form a basis of $\HH_{Eis}^1(\Gamma_{1}(N); \mathbb{C})$.

We can define the analogue of the Rademacher homomorphism for $\Gamma$ 
$$ 
\phi_E: \gamma \mapsto \int_{\gamma \cF}d\left( * \omega_E \right);
$$
where $E$ is an Eisenstein series of weight two on $\Gamma$. We strongly believe there are connections with Szech cocycles as discussed in \S~\ref{Szech}.  
These are generalizations of period functions 
defined for the full subgroup (cf. \S~\ref{Bianchidefinition} Theorem~\ref{thm1}). Our belief comes from the inner product formula proved above.

 We now list some properties of the Eisenstein series that we believe are connections between generalized Dedekind sums and period integrals of the Eisenstein series. 
 \begin{enumerate}
     \item Ito  ~\cite{MR870309} showed that (see also Weselmann ~\cite{MR953667}) up to a coboundary, the cocycles of Sczech are integrals of closed harmonic differential forms.
     \item These differential forms are given by certain Eisenstein series defined on the hyperbolic space $\mathbb{H}_3$.
     \item  Following Ito, we form an Eisenstein series $E_{(u, v)}(\tau, s)$ for $(\tau, s) \in \mathbb{H}_3\times \mathbb{C}$ with values in $\mathbb{C}^{3}$ associated to each cusp of $\Gamma_1(N)$. As a function of $s, E_{(u, v)}(\tau, s)$ can be analytically continued to all of $\mathbb{C}$.
     \item  Harder ~\cite{MR633658} shows that the differential $1$-form on the hyperbolic $3$-space induced by $E_{(u, v)}(\tau, s)$ is closed for $s=0$. Ito 
~\cite{MR870309} showed that the cocycle given by the integral of this closed differential $1$-form differs from the cocycle $\Psi(u, v)$ of Sczech by a coboundary.
 \end{enumerate}

\section{Lefschetz numbers for complex conjugations}
\label{Lefschetz numbers}

Recall the study  Lefschetz fixed point theorem for a congruence subgroup following Seng\"un-T\"urkelli ~\cite{MR3464620}.

\subsection{Lefschetz number of $\sigma$ for $\Gamma_1(N)$}
Let $\Gamma = \Gamma_1(N) \leq \SL_2(\mathcal{O}_K)$ be a congruence subgroup of level ($N$). For sufficiently large N (for example, if $N > 3$), the congruence subgroup $\Gamma_1(N)$ is torsion-free (cf. ~\cite[\S5.1]{MR2887610}). In this section, we use ~\cite[Theorem 2.1]{MR3464620} to calculate the Lefschetz numbers for the 
congruence subgroups of the form $\Gamma_1(N)$. 

Let $\operatorname{tr}\sigma^i$ is the trace of the map $\sigma^i:\HH^i(X^{BS}_\Gamma; \mathcal{M}_k) \rightarrow \HH^i(X^{BS}_\Gamma; \mathcal{M}_k)$ induced by $ \sigma$
on the $i$ th cohomology groups. Recall that Lefschetz number for $\sigma$ is defined as
\begin{center}

 $L\left(\sigma, \Gamma_1(N), M_{k}\right)$:= $\sum\limits_{i=0}^{\infty} (-1)^i  \operatorname{tr} \left( \sigma^i \right).$  
\end{center}
 Let $\chi((X^{BS}_\Gamma)^\sigma)$ be the Euler-Poincar\'e characteristic of the fix point set $(X^{BS}_\Gamma)^\sigma$ of $\sigma$ acting on $X^{BS}_\Gamma$.

Then $X_1(N)^{\sigma}$ consists of translations of surfaces $F\left(\gamma_{1}\right)$ and $F\left(\gamma_{1}^{\prime}\right)$ and the number of translations of these surfaces are denoted by $A_1(N)$ and $B_1(N)$.

%Let $D$ be the discriminant of $K / \mathbb{Q}$ and let $t$ be the number of distinct prime divisors of $D$. 
The set of places of $K$ will be denoted by $V$, and $V_\infty $ (resp. $V_f$) refers to the set of archimedean (resp. non-archimedean) places of $K.$
%Let $(N)=$ $\prod_{p \mid D} p_{p}^{j_{p}} \prod_{p \nmid D}(p)^{j_{p}}$ be %an ideal with $N>2$, and let $\Gamma=\Gamma_1(N)$ be the  congruence subgroup of level $(N)$. Let $s=\#\{p$ prime $|p| D, p \neq 2$ and $\left.j_{p} \neq 0\right\} .$

For all finite places $v \in V_f$,  let $\mathfrak{p}_v$ be the prime ideal corresponding to $v$ and $N_v:= [K_v:\mathbb{Q}_p]$ if $v|p$.
Consider the congruence subgroup $\Gamma= \Gamma_1(N) \subset \SL_2(\mathcal{O}) $ for an ideal $N\mathcal{O} =\prod \mathfrak{p}_v^{j_v}$. The completion $\Gamma_v\left(j_v\right)$ of  $\Gamma (=\Gamma_1(\mathrm{N}))$ in $\SL_2\left(\mathcal{O}_v\right), v \in V_f$, is a  subgroup of
$\SL_2\left(\mathcal{O}_v\right)$.
 Then one has $\mathfrak{p}_v \cap \mathbb{Z}=p \mathbb{Z}$ for some prime $p$, and $p \mathcal{O}_v=\mathfrak{p}_v^{e_v}$. Define
$$
s_v:=\left[\frac{e_v}{p-1}\right]+1 .
$$
Let $M(s)$ =$\begin{pmatrix}
  a & b\\
  c & d
\end{pmatrix}$ , $s \in \mathbb{N}$, denote the set of $(2  \times 2 )$-matrices with $a, c , d$ in $\mathfrak{p}_v^s$ and $b$ in $\mathcal{O}_v$. For all $s \geqslant s_v$, the exponential map defines a bijection exp: $M(s) \rightarrow 1+M(s)$. Then the map exp induces for any $s \geqslant s_v$ a bijection
$$
\mathfrak{sl}_2\left(\mathcal{O}_v\right) \cap M(s) \stackrel{\sim}{\rightarrow} \Gamma_v(s)
$$
where $\mathfrak{sl_2\left(\mathcal{O}_v\right)}$ denotes the $\mathcal{O}_v$-Lie algebra of $\SL_2\left(\mathcal{O}_v\right)$. 

Now, we are generalizing the results mentioned in the work of Rohlfs-Schwarmer ~\cite{MR1480883} from subgroups of symplectic groups to subgroups of the Bianchi group $\SL_2(\mathcal{O})$.

\begin{lemma}
Local cohomology groups for a subgroup $\Ga \leq \SL_2(\sO)$ can be computed as follows:
\begin{enumerate}
\item
If $\mathfrak{p}_v \mid p$ and if $p>2$ then $\HH^1\left(\sigma; \Gamma_v\left(j_v\right)\right)=\{1\}$ for $j_v=0,1,2, \ldots$
\item
If $\mathfrak{p}_v \mid 2$ and if $j_v \geqslant s_v+e_v$ then cardinality of $\HH^1\left(\sigma; \Gamma_v\left(j_v\right)\right)=2^{ N_v}$ for some $N_v \in \N$. The inclusion $\Gamma_v\left(j_v\right) \rightarrow$ $\Gamma_v\left(j_v-e_v\right)$ induces the trivial map
$$
\HH^1\left(\sigma; \Gamma_v\left(j_v\right)\right) \rightarrow \HH^1\left(\sigma; \Gamma_v\left(j_v-e_v\right)\right) .
$$
\end{enumerate}
\end{lemma}
\begin{proof} 

The proof is exactly the same as in loc. cit. In this case, we use the unipotent subgroup
$U(\mathcal{O}_v/2\mathcal{O}_v):= \Bigg\{ \left(\begin{array}{cc}
a & b \\
c & -a
\end{array}\right) \in M_2(\mathcal{O}_v) \Bigg|  a,c \in \mathcal{O}_v/2\mathcal{O}_v$  $\text{and}$ $b \in \mathcal{O}_v \Bigg\} $ instead.

\end{proof}
 \begin{lemma}
 \label{AB}
 Let $\sigma$ and $\Ga$ be as above. For all $\gamma \in \{\gamma_1,\gamma'_1\}$,  the cohomology groups $\HH^1(\sigma;  ^{\gamma}\Ga)$ are 
 $2$ groups. 
 \end{lemma}
 \begin{proof}
We recall the  relation from Rohlfs ~\cite[page $201$]{MR507801}:

  $$
\HH^1(\sigma; \Gamma)= \prod_{v}\HH^1(\sigma;  \Gamma_{v}(j_v)).
$$

Using $\HH^1\left(\sigma;  \Gamma_v\left(j_v\right)\right)=2^{ N_v}$, we get

$$
\HH^1(\sigma; \Gamma)= \prod_{v}\HH^1(\sigma;  \Gamma_{v}(j_v)) = \prod_{v} 2^{N_v}.
$$

 The  $\gamma$- twisted $\sigma$ action on $\Gamma $ produces the following short exact sequence:
 \[
0 \rightarrow {}^{\gamma}\Gamma \rightarrow   \Gamma \rightarrow \frac{\Gamma}{ {}^{\gamma}{\Gamma} }\rightarrow 0.
\]

This induces an exact sequence:

\[
\ldots \HH^{0}\left(\sigma ; \frac{\Gamma}{ {}^{\gamma}{\Gamma} }\right) \rightarrow \HH^{1}\left(\sigma; {}^{\gamma}\Gamma \right) \xrightarrow{\delta} \HH^{1}\left(\sigma ; \Gamma \right) \rightarrow \HH^{1}\left(\sigma ; \frac{\Gamma}{ {}^{\gamma}{\Gamma} }\right) \rightarrow \ldots
\]
From basic group theory, we have 
$$
\frac{\HH^1(\sigma; {}^{\gamma}\Gamma)}{\text{Ker}\delta} : = \text{Im} \delta.
$$
Since $\sigma$ is an involution (order $2$) so $\HH^{0}\left(\sigma ; \frac{\Gamma}{ {}^{\gamma}{\Gamma} }\right)$ is a $2$ group. 
We also note that $\text{Im} \delta \subset \HH^1(\sigma;\Gamma)$ and hence a $2$ group. 
We deduce that  that $ \#\HH^1(\sigma; {}^{\gamma}\Ga)=2^r$  for some $r \in \N \cup \{0\}$. 
\end{proof}
We define the following important quantities that we use in our computations:
\begin{align}
A_1(N) := \#\HH^1(\sigma; {}^{\gamma_1}\Ga) = 2^a\\
B_1(N) :=  \#\HH^1(\sigma; {}^{{\gamma_1}^\prime} \Ga)  = 2^b;
\end{align}
for $a, b \in \N \cup\{0\}$.

We calculate the Lefschetz numbers for $\Gamma_1(N)$ using  ~\cite[Theorem $2.1$]{MR3464620}.

 \begin{prop}
 
 \label{Lefschetzkey}
 Consider the subgroup  $\Gamma_1(N)$ and let $ A_1(N), B_1(N)$ be as above. Define
 \[
 C_1(N):=B_1(N)). \frac{-N^{2}}{12} \prod\limits_{p \mid N}\left(1-p^{-2}\right) \cdot(k+1) .
 \]
 The Lefschetz number is given by
$$
L\left(\sigma, \Gamma_1(N), M_{k}\right)= \begin{cases}A_1(N)+ ( \frac{N+2}{2}) C_1(N) & \text { if } N \text { is even }\\
A_1(N)+ (\frac{N+1}{2}) C_1(N) & \text { if } N \text { is odd. }\end{cases}.
$$
\end{prop}
\begin{proof}
 The proof is exactly the same as in loc. cit. Here, we have this for $\Gamma_1(N)$
 \[
\chi\left(\Gamma_1(N)^{\gamma_{1}^{\prime} \sigma}\right)= \begin{cases} (\frac{N+2}{2}) \chi\left(\Gamma^e_1(N)\right) & \text { if } N \text { is even }, \\ (\frac{N+1}{2}) \chi\left(\Gamma^e_1(N)\right) & \text { if } N \text { is odd }.\end{cases}
\] 
and 
\begin{align*}
\chi\left(\Gamma^e_1(N)\right):=\chi\left(Y_1(N)\right)=
\chi\left(X_1(N)\right)-\#\left\{\right. \text{cusps of} \left.Y_1(N)\right\}
\\= \: (-1 / 12) N^{2}\prod_{p \mid N}\left(1-p^{-2}\right).
\end{align*}
Here, $\Gamma^{e}_{1}(N)$ is the congruence subgroup of the elliptic modular group $\mathrm{SL}_{2}(\mathbb{Z})$ of level $N$.

\end{proof}

\begin{corollary} 
\label{cor5.1}
Let $p$ be an odd rational prime that is unramified over $K .$ Let $t$ be the number of distinct prime divisors of $D$. 
For $n>0$ we have
$$
L\left(\sigma, \Gamma_1\left(p^{n}\right), M_{k}\right)= {-(2^a + ({p^n+1}) 2^{b-1})}. \frac{\left(p^{2n}-p^{2n-2}\right)}{12} \cdot(k+1).
$$
\end{corollary}
\begin{proof} 
Recall that $p$ is the only prime divisor in the level.  
Since the prime is unramified, we have
\begin{align*}
(A+ (\frac{N+1}{2}) B). \frac{-N^{2}}{12} \prod_{p \mid N}\left(1-p^{-2}\right) \cdot(k+1)\\
= (2^a + (\frac{p^n+1}{2}) 2^b). \frac{-p^{2n}}{12} \left(1-p^{-2}\right) \cdot(k+1)\\
=(2^a + ({p^n+1}) 2^{b-1}). \frac{-1}{12} \left(p^{2n}-p^{2n-2}\right) \cdot(k+1)\\
 = {-(2^a + ({p^n+1}) 2^{b-1})}. \frac{\left(p^{2n}-p^{2n-2}\right)}{12} \cdot(k+1).
\end{align*}
\end{proof}
\section{Traces on the Eisenstein cohomology for $\Gamma_1(N)$}
\label{Tracecohomo}
We compute the trace of complex conjugation on the Eisenstein cohomology following  Seng\"un-T\"urkelli ~\cite{MR3464620} for the subgroups of the form $\Ga_1(N)$ inside the Bianchi modular groups.

Let $\sigma$ represents the {\it complex conjugation} and let the {\it twisted complex conjugation} be denoted by $\tau$. The actions of both $\sigma$ and $\tau$ on $\mathbb{H}_3$ and $\mathrm{SL}_2(\mathcal{O}_K)$ are defined in~\cite{MR3464620}.

  We use the symbol $\rho$ when we want to state results that are true for {\it both} of them. 
Consider the number  $\alpha=\#\left((\frac{\Gamma_1(N)}{\Gamma(N)})^\sigma\right)$ and $\delta$ be the Kronecker $\delta$-function. 
Recall that we are working with an imaginary quadratic field $K$ of class number one.

Let $t$ be the number of distinct prime divisors of the discriminant of $K/\mathbb{Q}$. For a positive odd number $N = p_{1}^{n_{1}} \ldots p_{\tau}^{n_{r}}$ with prime divisors $p_{i}$ that are unramified in $K$, let $\Gamma_1(N)$ be the congruence subgroup of the Bianchi group $\SL_{2}(\cO_K)$ with level the ideal  $(N)$. Let $\sigma$ be the involution induced in the Eisenstein cohomology of $\Gamma_1(N)$ by non-trivial
automorphisms of $K$.  
The following theorem computes the trace of $\sigma$ for Eisenstein cohomology groups:
\begin{thm}
\begin{enumerate}

    \item \label{cohomology1}
    Trace of $\sigma$ on the first cohomology group is given by:
    \[
    \operatorname{tr}\left(\sigma \mid \HH_{Eis}^{2}\left(\Gamma_{1}(N); \mathcal{M}_{k}\right)\right)=-2^{t-\tau-1} \cdot \alpha \cdot \prod\limits_{i=1}^{\tau}\left(p_{i}^{2 n_{i}}-p_{i}^{2 n_{i}-2}\right)+\delta(0, k).
    \]
\item \label{cohomology2} Let $p$ be a rational prime that is inert in $K$. The trace of $\sigma$ on the second cohomology group is given by:
$$
\operatorname{tr}\left(\sigma \mid \HH_{Eis}^1\left(\Ga_1\left(p^n\right); \mathbb{C}\right)\right)= \begin{cases}-2, & \text { if } n=1 \\ -\#\mathcal{C}_{p^n} . \left(\frac{-2}{p^2-1} \right), & \text { if } n>1 .\end{cases}
$$
    \end{enumerate}
\end{thm}
In particular, the trace of ${\sigma}^{2}$ on $\HH_{Eis}^{2}\left(\mathrm{SL}_{2}(\cO_K); \mathcal{M}_{k}\right)$ is
$-2^{t-1}+\delta(0, k)$. 
\\
\begin{proof}
To prove of $(1)$ in the above theorem, we follow the same strategy as in the loc. cit. It is enough to compute
$$
 \#\left( ( (\Ga \slash \Gamma(N)) \cdot  (\Gamma(N) \backslash G \infty)\right)^{\rho}
 = \#\left((\Gamma  \slash \Gamma(N))^\rho \right) \cdot  \#( \Gamma(N) \backslash G \infty)^{\rho}
 $$
 
\; \; \; \; $= \alpha\cdot \#\left(U^{+}(R) \backslash \mathrm{SL}_{2}(R)\right)^{\rho}
$ \  where $\alpha $ is equal to $\#\left((\Gamma  \slash \Gamma(N))^\rho \right)$.

 It is enough to compute $\#\left(U^{+}(R) \backslash \mathrm{SL}_{2}(R)\right)^{\rho}$ (cf.  Seng\"un-T\"urkelli ~\cite{MR3464620}).
    Observe that $\sigma$ fixes the ideal $(N)$. Therefore, the action of $\sigma$ on $\cO_K$ descends to the action on $R$. The action of $\sigma$ and $\tau$ on $\mathrm{SL}_{2}(R)$ as in ~\cite[p. 251]{MR3464620}.

We first treat the special situation when $N=p^{n}$ for a prime $p$. Consider the bijections of sets of matrices:
\[
U^{+}(R) \backslash \mathrm{SL}_{2}(R) \simeq U^{+}(R) \backslash B(R) \times B(R)   \backslash \mathrm{SL}_{2}(R);
\]
Where $B(R)$ is the subgroup of upper-triangular matrices, there are well-known bijections.
$$
B(R) \backslash \mathrm{SL}_{2}(R) \leftrightarrow \mathbb{P}^{1}(R), \quad\left[\begin{array}{cc}
a & b \\
c & d
\end{array}\right] \mapsto(a: c).
$$
Let $\mathbb{P}^{1}(R)$ denotes the projective line over $R$, and
\[
U(R) \backslash B(R) \leftrightarrow R^{*}, \quad\left[\begin{array}{cc}
a & b \\
0 & a^{-1}
\end{array}\right] \mapsto a.
\]
These bijections lead to the identification:
\[
U^{+}(R) \backslash \mathrm{SL}_{2}(R) \simeq R^{*} /\{\pm 1\} \times \mathbb{P}^{1}(R) .
\]
It is straightforward to transfer the action of $\sigma$ and $\tau$ to the right-hand side. We immediately see that:
$$
\left(U^{+}(R) \backslash \mathrm{SL}_{2}(R)\right)^{\rho} \simeq\left(R^{*} /\{\pm 1\}\right)^{\rho} \times \mathbb{P}^{1}(R)^{\rho}.
$$
Start with computing $\# \mathbb{P}^{1}(R)^{\rho}$. It can be seen that $\mathbb{P}^{1}\left(R^{\sigma}\right) \hookrightarrow$ $\mathbb{P}^{1}(R)$ and in fact $\mathbb{P}^{1}\left(R^{\sigma}\right)=\mathbb{P}^{1}(R)^{\sigma}$. 
Note that $\mathbb{P}^{1}\left(R^{\sigma}\right) \simeq \mathbb{P}^{1}\left(\mathbb{Z} / p^{n} \mathbb{Z}\right)$
 and thus has cardinality $p^{n}+p^{n-1}$. Computation shows that $\mathbb{P}^{1}(R)^{\tau}$ has the same number of elements.

Let us now consider $\#(R^*) { }^{\rho}$. The action of $\sigma$ and $\tau$ are same on $R^{*}$. Clearly we have $\#\left(R^{*} /\{\pm 1\}\right)^{\rho}=(1 / 2) \cdot \#\left(R^{*}\right)^{\rho}$.
\begin{itemize}
    \item Let  $p$ be a prime that splits in $K$. We then have 
\[
 \left(R\right)^{\rho}=\left({\cO_K} / p^{n} {\cO_K}\right)^{\rho}=\left(\mathbb{Z}\sqrt{-d} / p^{n} \mathbb{Z}\sqrt{-d}\right)^{\rho} = \left(\mathbb{Z}\sqrt{-d}\right)^{\rho} /\left( p^{n} \mathbb{Z}\sqrt{-d}\right)^{\rho}. 
 \]
 \begin{enumerate}
     \item 
     When $\rho =\sigma$,  the automorphism $\sigma$ acts on $\mathbb{Z}\sqrt{-d}$ by $\sigma (a + b\sqrt{-d}) = a-b\sqrt{-d}$.
Let $(a + b\sqrt{-d}) \in \left(\mathbb{Z}\sqrt{-d}\right)^{\sigma}$ then $(a + b\sqrt{-d}) = (a - b\sqrt{-d})$, this implies  $b=0$  and $ a \in \mathbb{Z} $.
\item 
On the other hand if  $\rho = \tau$, the automorphism $\tau$ acts on $\mathbb{Z}\sqrt{-d}$ by $\tau (a + b\sqrt{-d}) = -a+b\sqrt{-d}$.
Consider now  $(a + b\sqrt{-d}) \in \left(\mathbb{Z}\sqrt{-d}\right)^{\tau}$ then $(a + b\sqrt{-d}) = (-a + b\sqrt{-d})$, this implies  $a=0$  and $ b \in \mathbb{Z} $.
We have $R \simeq\left(\mathbb{Z} / p^{n} \mathbb{Z}\right)^{2}$ and hence $\#\left(R^{*}\right)^{\rho}=\#\left(\mathbb{Z} / p^{n} \mathbb{Z}\right)^{*}=p^{n}-p^{n-1}$.
 \end{enumerate}

   \item 
   When $p$ is inert in $K$, we can view $R$ as the quadratic extension $\left(\mathbb{Z} / p^{n} \mathbb{Z}\right)[\omega]$ of the ring $\mathbb{Z} / p^{n} \mathbb{Z}$. It follows that $R^{*}=\{a+b \cdot \omega 
 \in \left(\mathbb{Z} / p^{n} \mathbb{Z}\right)^{*}[\omega] \}$.  We calculate $\left(R\right)^{\rho}$ when $p$ is split in $K$.  We deduce that  $\left(R^{*}\right)^{\rho}$ is given by 
 \begin{center}
 \[
 \{a+b \cdot \omega \in  \left(\mathbb{Z} / p^{n} \mathbb{Z}\right)[\omega] \mid p \nmid a, b=0 \ \text{for} \ \rho = \sigma \ \text{or} \ p \nmid b, a=0 \  \text{for} \ \rho = \tau.  \}
 \]
 \end{center}
The cardinality of this set is  $p^{n}-p^{n-1}$.
   \end{itemize}
   In both inert and split cases, we get the quantity $\#\left(U^{+}(R) \backslash \mathrm{SL}_{2}(R)\right)^{\rho}=(1 / 2) \cdot\left(p^{2 n}-p^{2(n-1)}\right).$

To finish the proof, let us assume that $N=p_{1}^{n_{1}} \ldots p_{r}^{n_{r}}$ is positive number whose prime divisors $p_{i}$ are unramified in $K$. The result in this general case follows from the simple fact that
\[
\SL_{2}(\cO_K /(N)) \simeq \mathrm{SL}_{2}\left(\cO_K /\left(p_{1}\right)^{n_{1}}\right) \times \ldots \mathrm{SL}_{2}\left(\cO_K /\left(p_{r}\right)^{n_{r}}\right)
\]

We now prove part $(2)$.
Recall the homomorphisms $\Psi(u,v)$ and $\mathbb{C}\left[\Psi_p^*\right]$ be as in   Seng\"un-T\"urkelli ~\cite[p. 252, 253]{MR3464620}.
The proof of (2) is again exactly the same as in loc. cit. The only difference is the order of the cuspidal subgroups. 

For $N=p$, we have  $\#\mathcal{C}_{p} = p^2 -1$. The trace of $\sigma$ on $\mathbb{C}\left[\Psi_p^*\right]$ is $\left(p^2 -1\right). \frac{-2}{p^2-1} = -2$.
For $N=p^n$ with $n>1$, we have  $\#\mathcal{C}_{p^n} =| \bar{P}_{+} \backslash \mathrm{SL}_2(\mathbf{\cO_K} / p^n \mathbf{\cO_K}) / \bar{P} |$. The trace of $\sigma$ on $\mathbb{C}\left[\Psi_p^*\right]$ is $\left(\#\mathcal{C}_{p^n} \right). \frac{-2}{p^2-1}$.
 
\end{proof}

\subsection{Proof of the main theorem}
\begin{proof} 
From  ~\cite[Proposition $2.2$]{MR3464620}, recall that
$$
\operatorname{dim} \HH_{c u s p}^1\left(\Gamma; M_{k}\right) \geq \frac{1}{2} \left(L(\sigma, \Gamma, M_k)+\operatorname{tr}\left(\sigma_{Eis}^1, \Gamma, M_k\right)-\operatorname{tr}\left(\sigma_{Eis}^2, \Gamma, M_k\right)\right) .
$$

The claim in (2) follows directly from theorem ~\ref{cohomology1} and theorem ~\ref{cohomology2}, together with the Lefschetz number formula provided in corollary  ~\ref{cor5.1}. 
\end{proof}
\begin{corollary}
Let all the parameters be as in the theorem above. We have the following weight aspect asymptotic dimension formula of cuspidal cohomology: 
$$
 \operatorname{dim} \HH_{cusp}^1\left(\Gamma_{1}\left(p^n\right); M_{k}\right) \gg k
$$
as $k$ increases and $n$ is fixed.
\end{corollary}
\begin{proof}
Fix the congruence subgroup  $\Gamma$.  By Proposition ~\ref{prop8}, the dimension of the Eisenstein part of the cohomology is the same for every weight $k>0$. Hence, the weight aspect of the asymptotic can be deduced from the corollary ~\ref{cor5.1} of the previous section.
\end{proof}

In a recent paper ~\cite{Fu}, Weibo Fu determined both the upper and lower bounds of cuspidal cohomology (weight asymptotic), using the work of Finis, Grunewald, and Tirao ~\cite{MR2649984}. Fu computed these bounds without utilizing the Eisenstein part of the cohomology. However, our calculation (similar to Seng\"un-T\"urkelli) shows that the Eisenstein part of the cohomology can also be used to make conclusions about the cuspidal part. Note that in loc. cit also, the author used both (completed) homology and cohomology groups and an interesting result of Marshall ~\cite{MR2912713} comparing the dimensions of the spaces over $\Q_p$ and $\C$.

\bibliographystyle{crelle}
\bibliography{ref}

\end{document}